%% file: PaperC-revision.tex
\documentclass{article}
\input{preamble.tex}
\begin{document}
\title{Coarse geometry of homeomorphism groups: Classifying countable Stone spaces}
\author{George Domat\footnote{Department of Mathematics, University of Michigan, Ann Arbor, MI, USA. ORCID ID: 0000-0003-1775-5349}, Hannah Hoganson\footnote{Department of Mathematics, University of Maryland, College Park, MD, USA. ORCID ID: 0009-0009-6933-5064} and Robert Alonzo Lyman\footnote{Department of Mathematics and Computer Science, Rutgers University--Newark, Newark, NJ, USA. ORCID ID: 0000-0002-8083-4747 \\ Corresponding author: \texttt{robbie.lyman@rutgers.edu}}}

\maketitle

\begin{abstract}
    Towards developing the tools of geometric group theory for non-locally compact topological groups, we give one of the first complete classifications of a family of such groups up to coarse equivalence, and when possible, up to quasi-isometry. 
    In a previous paper, we placed the homeomorphism groups of countable Stone spaces into three classes: coarsely bounded, unbounded yet generated by a coarsely bounded set, and unbounded but not generated by any coarsely bounded set. 
    Now we show that these are the coarse equivalence classes: Any two groups within one of these classes are in fact coarsely equivalent. 
    
    Furthermore, we show that groups in the second class are quasi-isometric to the \emph{Hamming cube,} the space comprising infinite binary sequences with finitely many nonzero entries equipped with the Hamming distance, and that groups in the third class are coarsely equivalent to the set of leaves of the regular one-ended tree of countably infinite valence.
    As part of the proof, we show that infinite Hamming graphs over finite alphabets are all bi-Lipschitz equivalent and prove a coarse geometric classification result for topological groups admitting exhaustions by proper, open, coarsely bounded subgroups.
\end{abstract}

\addtocontents{toc}{\protect\setcounter{tocdepth}{0}}
\section{Introduction}

Countable Stone spaces are classified up to homeomorphism by two pieces of data: their Cantor--Bendixson rank, which is a countable ordinal $\alpha$, and an integer $n \ge 1$, which is the number of points of maximal rank. 
We write $X_{\alpha,n}$ to denote the corresponding Stone space and study its group of symmetries, $\Homeo(X_{\alpha,n})$, which is a topological group equipped with the compact-open topology.
In \cite[Theorem A]{BDHL2025}, we showed that these groups all have metrizable coarse structures that fall into three categories: coarsely bounded, unbounded and generated by a coarsely bounded subset, or unbounded but not generated by any coarsely bounded subset. In this paper we show that these are also the coarse equivalence classes.

\begin{customthm}{A}
  \label{thm:mainclassification}
  Let $X_{\alpha,n}$ be a countable Stone space with $\alpha>0$.
  The groups $\Homeo(X_{\alpha,n})$ equipped with the compact-open topology fall into three coarse equivalence classes, which are determined by $\alpha$ and $n$ as follows.
  \begin{enumerate}
      \item If $n=1$, then for all $\alpha$, $\Homeo(X_{\alpha,1})$ is coarsely bounded and hence quasi-isometric to a point.
      \item If $n>1$ and $\alpha$ is a successor ordinal, then $\Homeo(X_{\alpha,n})$ is quasi-isometric to the countably infinite Hamming cube.  
      \item If $n>1$ and $\alpha$ is a limit ordinal, then $\Homeo(X_{\alpha,n})$ is coarsely equivalent to the set of leaves of the regular one-ended tree of countably infinite valence.
  \end{enumerate}
\end{customthm}

In particular, all unbounded groups that are generated by a coarsely bounded subset fall into the second category and are quasi-isometric to each other. This gives perhaps the first complete quasi-isometric classifications for an infinite family of non-locally compact topological groups.

Moreover, our result also supplies explicit geometric models. 
The \emph{countably infinite Hamming cube} is the graph whose vertices correspond to binary sequences with finitely many $1$'s. Two vertices are connected by an edge if they differ in a single coordinate. Geometrically, this is the $1$-skeleton of a connected component of an infinite-dimensional cube. It includes all of the finite Hamming cubes as subgraphs. The \emph{regular one-ended tree of countably infinite valence}
has countably infinitely many leaf vertices,
all non-leaf vertices have countably infinite valence,
and every pair of geodesic rays in the tree has a common tail.

Prior to this work, explicit descriptions of coarse structures associated to non-locally compact topological groups were only known in the case of automorphism groups of locally infinite trees and the isometry group of Urysohn space \cite[Section 6.5]{Rosendal}, homeomorphism groups of compact manifolds \cite{MR2018} and of certain fractals \cite{DuchesneTarocchi}, and the mapping class group of the plane minus a Cantor set \cite{CalegariBlog,Bavard2016,MR2023,SC2024}. 

As a corollary of our work, we have the following.

\begin{customcor}{B}
    Let $X_{\alpha,n}$ be a countable Stone space.
    The asymptotic dimension of the group $\Homeo(X_{\alpha,n})$ equipped with the compact--open topology is infinite if $\alpha$ is a successor ordinal and $n > 1$ and is $0$ otherwise. Furthermore, when $\alpha$ is a successor ordinal and $n > 1$, the group $\Homeo(X_{\alpha,n})$ has infinite coarse rank. 
\end{customcor}

\begin{proof}
    When $n > 1$ and $\alpha$ is a limit ordinal, the asymptotic dimension is $0$ because the group is coarsely equivalent to the set of leaves of a tree.
    When $n > 1$ and $\alpha$ is a successor ordinal, by the second statement of \Cref{thm:mainclassification} it suffices to prove that $\Homeo(X_{1,2})$ (see below) has infinite coarse rank. 
    This follows, mutatis mutandis, from arguments using commuting shift maps given in \cite[Section 9.2]{DHK2023} and \cite[Sections 5 \& 6]{GRV2021}.
    Alternatively, with some work, one can directly construct quasi-isometric embeddings of arbitrarily large flats into the countably infinite Hamming cube.
    In all other cases, the group $\Homeo(X_{\alpha,n})$ is coarsely bounded.
\end{proof}

The third statement of \Cref{thm:mainclassification} is a consequence of a more general result.

\begin{customthm}{C} \label{thm:maintrees}
    Let $G$ be a topological group such that $G = \bigcup_{k=1}^{\infty} G_k$, where each $G_k$ is a proper, open subgroup coarsely bounded in $G$. There exists a one-ended tree $T$ so that $G$ equipped with its canonical coarse structure is coarsely equivalent to the set of leaves of $T$ equipped with the subspace metric.
\end{customthm}

This result provides coarse equivalences between \emph{a priori} distinct groups. For example, this theorem recovers the results of \cite[Sections 6 \& 9]{DHK2023} on certain pure mapping class groups of infinite graphs. We also state and prove an analogue of \Cref{thm:maintrees} for topological groups $G$ which are countably generated over a proper, open subgroup coarsely bounded in $G$, which may be of independent interest; see \Cref{DEF:TreeofCB} and \Cref{prop:orbit-map-CE}. 

\subsection{A guiding example}

The Stone space $X_{1,2}$ is homeomorphic to the 2-point compactification of $\Z$, denoted $\hat{\Z}=\Z\cup \{\pm \infty\}$. 
The points $\pm \infty$ are what will be called \emph{maximal points} because they are accumulation points, while all other points are isolated in the topology. 
To illustrate the main ideas of this paper, we describe a graph $\G$ that captures the quasi-isometry type of $\Homeo(\hat{\Z})$ and is isomorphic to the Hamming cube. A general proof is done rigorously in \Cref{sec:CARtoHam}.

The vertices of $\G$ correspond to partitions of $\hat{\Z}$ into two clopen sets separating $-\infty$ and $+\infty$. Two vertices are connected by an edge if one partition can be constructed from the other by moving a singleton. 
The prototypical example of adjacent vertices $\mathcal{P}$ and $\mathcal{R}$ is given by 
\[
P_- = \{ x \le 0 \} \quad\text{and}\quad P_+ = \{x > 0\} \qquad\text{while}\qquad R_- = \{ x < 0\} \quad\text{and}\quad R_+ = \{ x \ge 0\}.
\]
Here $P_{-}=R_{-}\sqcup \{0\}$ and $P_{+}\sqcup \{0\}=R_+$. 
These are basic examples of what we call \emph{dividing partitions} and what it means for two partitions to \emph{differ by a shift}. See \Cref{fig:Gamma2_partitions} for an example of four vertices in $\Gamma$ and the corresponding edges between them. 

\begin{figure}[ht!]\centering
        \def\svgwidth{\textwidth}
        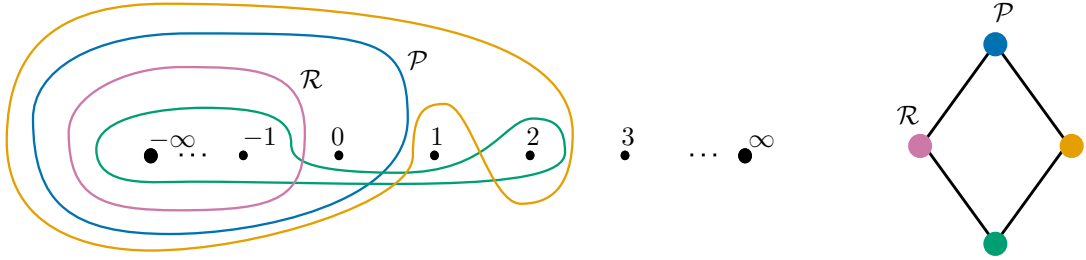
		\caption{A small piece of the graph $\Gamma$. The partitions $\mathcal{P}$ and $\mathcal{R}$ defined above are labeled in blue and pink, respectively. Partitions are represented as Jordan curves in the plane so that the interior of the curve contains $-\infty$.}
        \label{fig:Gamma2_partitions}
\end{figure} 

Because homeomorphisms preserve the set of maximal points, the group $\Homeo(\hat{\mathbb{Z}})$ acts on $\Gamma$; in fact it acts continuously. A homeomorphism $f$ acts by sending a partition $\mathcal{P} = P_- \sqcup P_+$ to the induced partition
\[ f.\mathcal{P} = \{ f(x) : x \in P_-\} \sqcup \{f(x) : x \in P_+\}. \]
In fact, in \cite{BDHL2025} we showed that $\G$ is a \emph{Cayley--Abels--Rosendal graph} for $\Homeo(\hat{\Z})$, meaning that the group has a canonical quasi-isometry type which is captured by the graph. 

Now we aim to see that $\G$ is isometric to a Hamming graph. Fix a basepoint partition, $\mathcal{P}_0$, say the partition with $P_{-}=\{x \leq 0\}$ and $P_+=\{x>0\}$ from above. 
Let $\mathcal{Q}=Q_{-}\sqcup Q_+$ be any other partition. Because the partitions are clopen sets, there is an $N >0$ such that \[Q_{-}\cap[-\infty,-N]=P_{-}\cap[-\infty,-N] \quad \text{and}\quad Q_+\cap[N,\infty]=P_+\cap[N,\infty].\] 
In particular, $\mathcal{P}_0$ and $\mathcal{Q}$ differ on finitely many points. Moving these points one at a time shows that the number of such points is exactly their distance in $\G$. In fact, recording the points that differ between $\mathcal{P}_0$ and $\mathcal{Q}$ allows us to label the vertices of $\Gamma$ by finite subsets of $\Z$. 
Two vertices are adjacent when their corresponding subsets differ by a single element. For example, the partition $\mathcal{R}$ given above corresponds to the subset $\{0\}$.
This viewpoint shows that $\Gamma$ is naturally isomorphic to the poset graph of finite subsets of $\Z$. See \Cref{fig:Gamma2_Zsubsets} for an illustration of this. 

\begin{figure}[ht!]\centering
        \def\svgwidth{.8\textwidth}
        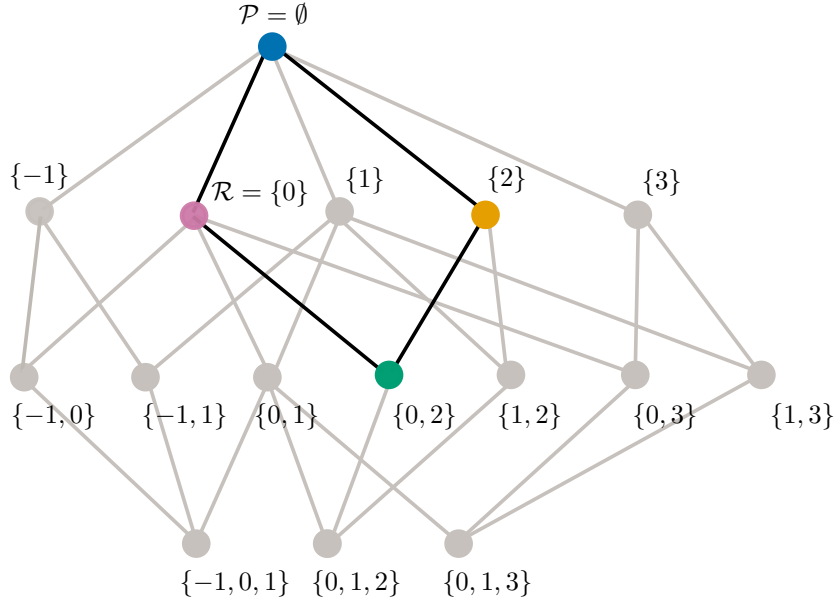
		\caption{A piece of the graph $\Gamma$ realized as the poset graph of finite subsets of $\Z$. The square piece highlighted in color is the square of $\Gamma$ pictured in \Cref{fig:Gamma2_partitions}.}
        \label{fig:Gamma2_Zsubsets}
\end{figure} 

The characteristic function of a subset $S \subset \mathbb{Z}$ allows us to think of vertices of the poset graph as $\mathbb{Z}$-indexed sequences in $\{0,1\}$.
In this perspective, an edge of $\Gamma$ corresponds to two sequences differing in one coordinate; thus $\Gamma$ is a Hamming graph.
See \Cref{fig:Gamma2_hamming} for a realization of a piece of $\Gamma$ as this Hamming graph. 
Finally, fixing any bijection $\Z \arr \omega$ allows one to reindex the vertex set and realize it as the standard countably infinite Hamming cube.

\begin{figure}[ht!]\centering
        \def\svgwidth{.7\textwidth}
        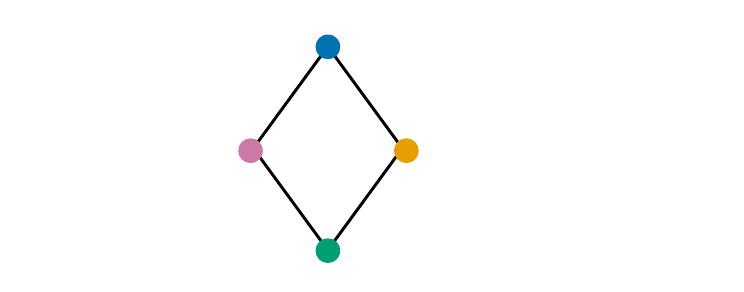
		\caption{A piece of the graph $\Gamma$ realized as the countably infinite Hamming cube. This shows the same square of $\Gamma$ pictured in \Cref{fig:Gamma2_partitions}. The check marks indicate the $0$th coordinate. }
        \label{fig:Gamma2_hamming}
\end{figure}

\subsection{History and motivation}

Much of modern geometric group theory has developed around Gromov's \cite{Gromov1987,Gromov1993} program of studying a group via its \emph{coarse geometric} properties. Traditionally this has focused on finitely generated (discrete) groups. This is because finitely generated groups admit well-defined quasi-isometry classes of left-invariant metrics, namely those given by the word metric for any (and hence all) finite generating sets. When moving to \textit{topological} groups, which need not be finitely generated, one can ask for a canonical class of left-invariant metrics that additionally generate the group topology, or at least are continuous with respect to it. Gromov's program was first extended to the class of locally compact topological groups (see e.g. \cite{CdH2016}). Here \emph{compact} generating sets and their associated word metrics yield the canonical quasi-isometry type, and although the word metric is always discrete, it is quasi-isometric to a metric that is compatible with the group topology.

More recently, Rosendal \cite{Rosendal}, building on work of Roe \cite{Roe}, has extended this framework to general topological groups. The substitute for finiteness or compactness here is the property of being \textit{coarsely bounded}. 
A subset of a topological group is \emph{coarsely bounded}, or CB, if it has finite diameter in any continuous (pseudo-)metric on the group. 
If a Polish group has a CB neighborhood of the identity, i.e. if it is \emph{locally CB}, then it has a metrizable coarse structure. Additionally, if a Polish topological group admits a CB generating set, the word metric associated to the closure of this generating set (and any analytic CB generating set) gives a well-defined quasi-isometry class of metrics on the group. 
Furthermore, this class has a representative that generates the given topology on the group. See also~\cite[Sections 1 \& 2]{BDHL2025} for a survey on this perspective. 

In our previous paper~\cite{BDHL2025}, we introduced \emph{Cayley--Abels--Rosendal graphs} for topological groups. Similar to the Cayley graph of a discrete group with respect to a finite generating set, Cayley--Abels--Rosendal graphs provide a concrete geometric model for the quasi-isometry type of a topological group that admits one. Geometric group theorists have long used Cayley graphs to great effect, and this paper continues in that tradition.

When a group $G$ has metrizable coarse structure but is not generated by a coarsely bounded set, no Cayley--Abels--Rosendal graph exists, but it is still desirable to have a geometric model for the coarse structure on the group. This was done for non-Archimedean Polish groups by Kopreski--Shaji~\cite{KopreskiShaji}, who introduced \emph{coarse Cayley--Abels--Rosendal} graphs, which are metric graphs with a discrete set of edge lengths. The vertex set of a coarse Cayley--Abels--Rosendal graph $\G$ for a group $G$, when equipped with the subspace metric, is coarsely equivalent to $G$. 

In \Cref{ssec:generaltrees}, building on Kopreski--Shaji's work and using Bass--Serre theory, we construct a countable graph that allows us to compute the coarse structure of $G$. This graph has infinitely many orbits of vertices, which are naturally organized by an exhaustive, countably infinite chain of proper, open subgroups. Nonetheless, the orbit of any vertex still captures the coarse structure of $G$.

\subsubsection{The discrete topology} The statement of \Cref{thm:mainclassification} specifies that $\Homeo(X_{\alpha,n})$ is equipped with the compact-open topology. For a topological group, when we write ``quasi-isometric,'' this should always be taken in reference to a maximal \emph{continuous} left-invariant (pseudo-)metric. Quasi-isometry is thus a \emph{topological group} invariant rather than an \emph{abstract} group invariant. Asking for the latter is equivalent to considering coarse structures, \`{a} la \cite{Rosendal}, on groups equipped with the \emph{discrete} topology. The notion of coarse boundedness (CB) in a discrete group is also referred to as \emph{strong boundedness} (SB). This property was first studied by Bergman \cite{Bergman2006} in the context of infinite symmetric groups (e.g. $\Homeo(X_{1,1})$) and has since been expanded on in \cite{Vlamis2026}. In particular, an SB-generated group  also admits a well-defined quasi-isometry class of metrics. In \cite{Vlamis2026}, building on work in \cite{BCMVVV2024}, Vlamis shows that when $\alpha$ is a successor ordinal, $\Homeo(X_{\alpha,n})$ is SB-generated. In fact, Vlamis's work shows that the coarsely bounded generating sets given in~\cite{BDHL2025} are strongly bounded. Thus, combining \Cref{thm:mainclassification} with \cite[Theorem 6.1]{Vlamis2026} we can obtain the following stronger result. 

\begin{customcor}{D}\label{cor:sbclassification}
    If $X_{\alpha,n}$ is a countable Stone space with $\alpha>0$ a successor ordinal and $n>1$, then $\Homeo(X_{\alpha,n})$ equipped with the discrete topology is quasi-isometric to the countably infinite Hamming cube. 
\end{customcor}

\subsubsection{Connections to surface mapping class groups}

One motivation for this project comes from attempts at obtaining a quasi-isometry classification for mapping class groups of infinite-type surfaces; see for example Problem 2.46(b) of the new Kirby problem list \cite{BKR2026}. 
When equipped with the (quotient of) the compact-open topology, these groups are no longer finitely generated and discrete, but instead are non-locally compact topological groups. By the work of Mann--Rafi \cite{MR2023}, it is known that many such mapping class groups are CB-generated and hence admit well-defined quasi-isometry classes of metrics. However, little is known about the quasi-isometry classification of these groups. So far, one can distinguish those mapping class groups that are themselves CB from those that are not~\cite{MR2023,RM2023} and some examples that are hyperbolic groups~\cite{SC2024}. Some quasi-isometry invariants such as asymptotic dimension, coarse rank, and divergence bounds are also known~\cite{GRV2021,KopreskiShaji,BNQR2026}. A full classification, however, is nowhere near in sight.

The end spaces of surfaces are second countable Stone spaces. In fact, every such Stone space is the end space of an infinite-type surface~\cite{Kerekjarto1923,Richards1963OnTC}. Mapping class groups act on their end spaces continuously by homeomorphisms. Therefore, given an infinite-type surface $S$ with end space $E(S)$ and subspace of ends accumulated by genus $E_g(S)$, we have the following short exact sequence of topological groups.
\begin{align*}
    1 \longrightarrow \PMCG(S) \longrightarrow \MCG(S) \longrightarrow \Homeo(E(S),E_g(S)) \longrightarrow 1
\end{align*}
where $\MCG(S)$ denotes the mapping class group, $\PMCG(S)$ denotes the closed subgroup of ``pure'' mapping classes, and $\Homeo(E(S),E_g(S))$ denotes the closed subgroup of $\Homeo(E(S))$ that preserves the closed subset $E_g(S) \subset E(S)$.  This suggests that a natural starting point for understanding the coarse geometric structure of $\MCG(S)$ is to first understand the coarse geometric structure of $\Homeo(X)$ for $X$ a second countable Stone space. 

Our classification does not imply an analogous quasi-isometric classification for mapping class groups. In \Cref{sec:surfaces}, we see that the quotient map given above is \emph{not} a quasi-isometry (\Cref{lem:forgetfulsurface}). 
Unlike in the group $\Homeo(X_{\alpha,n})$, we also show that the Cantor--Bendixson derivative on end spaces does \emph{not} induce a quasi-isometry between mapping class groups (\Cref{lem:canbendsurface}), even when $E_g(S) = \varnothing$. These results, together with \Cref{thm:mainclassification}, suggest that one \emph{must} make use of the topology/geometry of the surface when classifying mapping class groups up to quasi-isometry. 

The results in \Cref{sec:surfaces} do not cover all infinite-type surfaces. Curiously, the arguments fail for so-called ``translatable'' or ``avenue'' surfaces (\`{a} la \cite{SC2024} and \cite{HQR2022}, respectively). One such example of these surfaces is $\Sigma_{1,2}$, the zero genus surface with end space homeomorphic to $X_{1,2}$. This raises the following question. Here $\zeta(\Sigma)$ is a measure of complexity as defined in \cite{BNQR2026}. 

\begin{customq}{1}
\label{q:surface}
    Let $\Sigma$ be a stable surface with $\zeta(\Sigma)=2$ whose mapping class group is CB-generated but not coarsely bounded. Is $\PMCG(\Sigma)$ coarsely bounded in $\MCG(\Sigma)$? 
    
    More specifically, let $\Sigma_{1,2}$ be the genus-zero surface with end space $X_{1,2}$. Is $\PMCG(\Sigma_{1,2})$ coarsely bounded inside of $\MCG(\Sigma_{1,2})$? 
\end{customq}

\subsubsection{Connections to graph mapping class groups}

Our results extend more directly to the setting of mapping class groups of locally finite, infinite graphs. We refer the reader to \cite{AKB2025} for the relevant definitions. The mapping class group of a tree is isomorphic (as a topological group) to the homeomorphism group of its end space, i.e. a second countable Stone space. In this sense, these mapping class groups can be seen as direct extensions of the groups studied here. Additionally, \cite{DHK2023,DHK2025} show that the pure mapping class groups of many graphs do admit a well-defined quasi-isometry class of metrics. 
We make use of \Cref{thm:maintrees} to see that a class of graphs studied in \cite[Section 6 \& Section 9]{DHK2023} has pure mapping class groups coarsely equivalent to some of the groups studied here. 
In particular, these graphs have corresponding pure mapping class groups that are all coarsely equivalent. We do not recall all of the definitions pertaining to graph mapping class groups and instead direct the reader to \cite{DHK2023}.

\begin{customcor}{E}\label{cor:graphs}
    Let $\lambda$ be a limit ordinal, $n>1$, and $\Gamma$ a locally finite graph with exactly one end accumulated by loops and infinitely many other ends. The groups $\Homeo(X_{\lambda,n})$ and $\PMCG(\Gamma)$ are coarsely equivalent. 
\end{customcor}
\begin{proof}
    Letting $\ell$ be the full length function constructed in \cite[Proposition 6.4 \& Lemma 9.6]{DHK2023}. For each $k \in \N$ let $G_k < \PMCG(\Gamma)$ be given by $\ell^{-1}([0,k])$. By \cite[Proposition 6.4 \& Lemma 9.7]{DHK2023} this defines an exhaustion of $\PMCG(\Gamma)$ by proper, open, CB subgroups. 
    Thus, we can apply \Cref{thm:maintrees}. 
    By construction, each $G_k$ has countably infinite index in $G_{k+1}$ so the tree given by \Cref{thm:maintrees} is the regular one-ended tree of countably infinite valence. 
    Thus we conclude that $\PMCG(\Gamma)$ and $\Homeo(X_{\lambda,n})$ are coarsely equivalent. 
\end{proof}

In addition to this explicit application one may also ask: To what degree do the general methods of this paper extend to graph mapping class groups? 

\begin{customq}{2}
    Let $\Gamma$ be a locally finite, infinite graph with isolated ends not accumulated by loops. Let $\Gamma'$ be the graph obtained from $\Gamma$ by forgetting all such isolated ends. Does this forgetful map induce a quasi-isometry between $\MCG(\Gamma)$ and $\MCG(\Gamma')$? 
\end{customq}

This forgetful map is an analogue of the Cantor--Bendixson derivative and by \Cref{cor:CantorBendixson-QI} we know that the answer to this question is ``yes'' when $\Gamma$ is a tree with countable end space. However, by \Cref{lem:canbendsurface}, we know that the answer to the analogous question in the surface case is ``no.''

\subsection{Organization and outline}
Here we provide both an outline of the paper and of the proof of \Cref{thm:mainclassification}. First, in \Cref{sec:background} we recall background notions of coarse geometry and aspects of our work in \cite{BDHL2025} necessary for this paper.

In \Cref{sec:derivative}, mostly using ideas from \cite{BDHL2025}, we show that the Cantor--Bendixson derivative on a countable Stone space induces a coarse equivalence of homeomorphism groups. For successor ordinals, this immediately allows one to reduce \Cref{thm:mainclassification} to the case where $\alpha = 1$. We tackle the rest of the successor ordinal case in \Cref{sec:successcase}. Here we first exhibit isomorphisms between the Cayley--Abels--Rosendal graphs built in \cite{BDHL2025} with infinite Hamming graphs over finite alphabets (\Cref{sec:CARtoHam}). Then we study the geometry of these Hamming graphs and show that they are all bi-Lipschitz equivalent (\Cref{sec:HammingCubes}), thus finishing the proof of \Cref{thm:mainclassification} for successor ordinals. 

Next we turn to limit ordinals in \Cref{sec:limitcase}. Since the corresponding homeomorphism groups are no longer CB generated, they do not admit Cayley--Abels--Rosendal graphs. 
When $\lambda$ is a limit ordinal, there is a family of proper, open subgroups $G_k \le \Homeo(X_{\lambda,n})$ that exhaust $\Homeo(X_{\lambda,n})$ and are coarsely bounded in $\Homeo(X_{\lambda,n})$.
Bass--Serre theory provides an action of $\Homeo(X_{\lambda,n})$ on a 1-ended tree with all vertex stabilizers given by the $G_k$.

Each subgroup $G_k$ has countably infinite index in $G_{k+1}$. By \Cref{thm:maintrees}, the set of leaves of the regular one-ended tree of countably infinite valence
equipped with the subspace metric
therefore models the coarse structure on $\Homeo(X_{\lambda,n})$
for every limit ordinal $\lambda$.


The methods in \Cref{ssec:generaltrees} also apply to groups $G$ that are countably generated over an open subgroup $G_0$ coarsely bounded in $G$.
In this case, adding generators one at a time produces the groups $G_k$ in the countable exhaustion of $G$.
We show that one can blow up the vertices of the Bass--Serre tree in order to produce a geometric model for the group even when the $G_k$ are not coarsely bounded in $G$, see \Cref{DEF:TreeofCB} and \Cref{prop:orbit-map-CE}.
These constructions are similar to the construction of coarse Cayley--Abels--Rosendal graphs of \cite{KopreskiShaji} and we make use of a number of their tools.

If $n = 1$, we showed in~\cite[Corollary 14]{BDHL2025} that the group $\Homeo(X_{\alpha,1})$ is coarsely bounded. Thus, groups in this class are all coarsely equivalent to a point, and we have finished the proof of \Cref{thm:mainclassification}. To conclude we discuss connections to surfaces in \Cref{sec:surfaces}.

\subsection*{Acknowledgments}
The authors thank Christian Rosendal, Nick Vlamis, Michael Kopreski, and Beth Branman for many helpful conversations. 
The authors thank Brian Udall for pointing out simplifications in \Cref{sec:limitcase} and related ideas leading to \Cref{thm:maintrees}.
Hannah Hoganson was partially supported by NSF grant DMS--2303365 (MSPRF). George Domat was partially supported by NSF grant DMS--2303262 (MSPRF).

\subsubsection*{AI Acknowledgment} 
The statements and the outline of proofs in \Cref{sec:HammingCubes} were obtained by prompting an LLM, ChatGPT 5.5-Pro, and suggested to us by Christian Rosendal. 
The proofs in \Cref{sec:HammingCubes} were written by the authors upon reading the output provided by this prompting and the authors take full responsibility for the content of the paper.

\tableofcontents 

\addtocontents{toc}{\protect\setcounter{tocdepth}{3}}

\section{Background}\label{sec:background}
This section contains background material necessary for what follows. There is a brief review of the machinery of Roe~\cite{Roe} and Rosendal~\cite{Rosendal} for coarse geometry on topological groups. We discuss ordinals to fix notation, then we review the spaces $X_{\alpha,n}$ and the graphical models for their homeomorphism groups constructed in \cite{BDHL2025}.

\subsection{Coarse geometry}

We direct the reader to Rosendal's book \cite{Rosendal} for more thorough background on coarse geometry for general topological groups (see also \cite[Sections 1 \& 2]{BDHL2025}, \cite{Robbie} or \cite{KopreskiShaji} for an abridged overview). We remind the reader only of the basics, partially to set notation.

Every topological group $G$ admits a canonical coarse structure (\`{a} la \cite{Roe}), the \emph{left-coarse structure,} by considering all possible continuous left-invariant (pseudo)-metrics on $G$; equivalently all continuous, isometric actions of $G$ on metric spaces. A subset $A \subset G$ is \emph{coarsely bounded}, or CB, if it is bounded in this coarse structure. This means that $A$ has bounded diameter in any continuous left-invariant pseudometric on $G$. See \cite[Proposition 2.15]{Rosendal} for other characterizations. 

Adding the assumption that $G$ is Polish (meaning separable and completely metrizable), the coarse structure on $G$ is metrizable if and only if $G$ has a coarsely bounded neighborhood of the identity, i.e. $G$ is \emph{locally CB} \cite[Theorem 2.38]{Rosendal}. 
If $G$ is locally CB, then, by separability of $G$, it is also countably generated over a CB neighborhood of the identity. That is, there exists a CB neighborhood of the identity $H$ in $G$ and a countable set $S$ in $G$ so that $H$ and $S$ algebraically generate $G$. 
Furthermore, $G$ admits a well-defined quasi-isometry class of metrics if and only if $G$ is generated by a coarsely bounded set \cite[Proposition 2.72 and Theorem 2.73]{Rosendal}, which in the case of Polish groups may be taken to be a CB neighborhood of the identity $H$ together with a \emph{finite} set $S$.

A map $f\colon (X,d_X) \to (Y,d_Y)$ between metric spaces is a \emph{quasi-isometry} when there exist positive constants $(L,K,C)$ such that for all $x$ and $y$ in $X$, we have

\[
\frac{1}{L} d_X(x,y) - K \le d_Y(f(x),f(y)) \le L d_X(x,y) + K
\]
and for each $y \in Y$ there exists $x \in X$ such that $d_Y(f(x),y) \le C$.

The latter condition above says that the map $f$ is \emph{cobounded.} Replacing the affine upper control $t \mapsto Lt + K$ in the right-most inequality with an arbitrary increasing function that tends to infinity, $\eta \colon \mathbb{R} \to \mathbb{R}$ of $d_X(x,y)$, gives the definition of a \emph{bornologous} map of metric spaces. Doing the same for the inequality $\frac{1}{L} d_X(x,y) - K \le d_Y(f(x),f(y))$, one defines an \emph{expanding} map.

A \emph{coarse equivalence} of metric spaces is a map that is bornologous, expanding and cobounded. 
Since our topological groups have metrizable coarse structure, these definitions suffice to characterize coarse equivalence. 
When a Polish group is CB generated, coarse equivalence turns out to recover the notion of quasi-isometry. Indeed, by \cite[Theorem 2.40]{Rosendal} and \cite[Proposition 2.57]{Roe}, when a Polish group is CB generated, its coarse structure can be realized by a quasi-geodesic metric space. 
By \cite[Lemma 2.66]{Rosendal} a coarse equivalence between quasi-geodesic metric spaces is in fact a quasi-isometry.

Here is one final general proposition that we will make use of repeatedly. 

\begin{prop}\cite[Proposition 4.37]{Rosendal}\label{prop:sesce}
    Let $K$ be a closed normal subgroup of a Polish group $G$. The quotient map $G \rightarrow G/K$ is a coarse equivalence if and only if $K$ is coarsely bounded in $G$. If $G$ is generated by a CB set, then the quotient map is a quasi-isometry.
\end{prop}

A \emph{Cayley--Abels--Rosendal} graph for a Polish group $G$ is a countable, connected graph $\Gamma$ on which $G$ acts continuously with coarsely bounded stabilizers, one orbit of vertices and finitely many orbits of edges. We introduced Cayley--Abels--Rosendal graphs in our previous paper~\cite{BDHL2025}. When $G$ admits a Cayley--Abels--Rosendal graph, $G$ is generated by an open, coarsely bounded set. The group $G$ and the graph $\Gamma$ are quasi-isometric.

\subsection{Ordinals}

An \emph{ordinal} is the order isomorphism class of a well-ordered set. The class of ordinals is \emph{itself} well-ordered. Therefore, if $\alpha$ and $\beta$ are ordinals, then we write $\beta<\alpha$ to say that $\beta$ occurs before $\alpha$ in the associated well-order. 

An ordinal $\alpha$ is a \emph{successor ordinal} if there is an ordinal $\beta$ with $\alpha=\beta+1$. That is, $\beta < \alpha$, and $\alpha$ is the smallest ordinal with this property; its existence is guaranteed by the well-ordering of ordinals.

A nonzero ordinal that is not a successor ordinal is called a \emph{limit ordinal}. We denote the first infinite ordinal by $\omega = \{0,1,2,\ldots\}$; it is also the first limit ordinal. We will not consider $0$ to be either a limit or successor ordinal. 

In this paper we only use countable ordinals. If $\lambda$ is a countable limit ordinal, then we can choose an increasing sequence $\alpha_0<\alpha_1<\dots <\lambda$ for which $\lim_{k\arr \infty} \alpha_k=\lambda$. 

\subsection{The spaces $X_{\alpha,n}$}

\begin{DEF}
    A \emph{Stone space} is a space that is compact, Hausdorff and totally disconnected.
\end{DEF}

Stone spaces $X$ are either perfect or contain an isolated point.
We are interested in Stone spaces $X$ with countably many points. The only perfect, countable Stone space is empty.

\begin{DEF}
    The \emph{Cantor--Bendixson derivative} of a space $X$, denoted $X'$, is the closed subset of $X$ remaining after all isolated points of $X$ are removed.
    More generally, the operation of Cantor--Bendixson derivative can be iterated by transfinite induction:
    \begin{enumerate}
        \item[] {\bf Base case:} $X^{(0)} = X$
        \item[] {\bf Successor case:} $X^{\alpha + 1} = (X^\alpha)'$
        \item[] {\bf Limit case:} If $\displaystyle\lambda = \lim_{k\to\infty} \alpha_k$, then
            $X^{(\lambda)} = \bigcap_{k} X^{(\alpha_k)}$.
    \end{enumerate}
\end{DEF}

If $X$ is a countable Stone space, then for some countable ordinal $\gamma$, the space $X^{(\gamma)}$ is empty.
By the finite intersection property,
the smallest such ordinal is a successor ordinal $\alpha + 1$,
and the space $X^{(\alpha)}$ is nonempty but has finitely many points, say $n$. 
These two pieces of data: the ordinal $\alpha$ and the number $n \ge 1$ characterize $X$ up to homeomorphism~\cite{MS1920}. We denote countable Stone spaces $X_{\alpha, n}$.
We call the points of $X_{\alpha,n}^{(\alpha)}$ the \emph{maximal points} of $X_{\alpha, n}$. The terminology originates from the Mann--Rafi preorder on ends of an infinite type surface defined in \cite[Section 4]{MR2023}.

A concrete model for $X_{\alpha,n}$ can be given by the ordinal $\omega^{\alpha}\cdot n + 1$ equipped with the order topology. However, we will not make use of this order structure and only consider $X_{\alpha,n}$ as a topological space. 

\subsection{Graphical models for $\Homeo(X_{\alpha, n})$} \label{subsec:CARgraphs}
In this subsection we summarize necessary constructions and results from our first paper, \cite{BDHL2025}.

\begin{lem}[\cite{BDHL2025} Corollary 14]
The group $\Homeo(X_{\alpha,1})$ is coarsely bounded in itself,
whence it is coarsely bounded in any topological group into which it embeds continuously.
\end{lem}

\begin{DEF}
A \emph{dividing partition} of $X_{\alpha, n}$ is a clopen partition
\[ \mathcal{P} = P_1 \sqcup \cdots \sqcup P_n \]
into $n$ sets, such that each $P_i$ contains a single maximal point. As such, in a dividing partition $\mathcal{P}$, each partition element $P_i$ is homeomorphic to $X_{\alpha,1}$. In our previous paper, these were called ``good partitions" and were defined only for successor ordinals.
\end{DEF}

The stabilizers of clopen partitions of $\Homeo(X_{\alpha, n})$ form a neighborhood basis of the identity~\cite[Section 1.2]{BranmanLyman}.

\begin{lem}\label{lem:stabCB}
The stabilizer in $\Homeo(X_{\alpha, n})$ of a dividing partition has an open, finite-index subgroup isomorphic to $\prod_{i=1}^n \Homeo(X_{\alpha,1})$. As such, it is an open, coarsely bounded subgroup.
\end{lem}

\begin{proof}
    This stabilizer is precisely the subgroup used in \cite[Corollary 15]{BDHL2025} to show that the group $\Homeo(X_{\alpha, n})$ is locally CB.
\end{proof}

\begin{DEF}\label{shifts}
Fix $X = X_{\alpha, n}$.
If $\beta < \alpha$ is a countable ordinal and $n > 1$, we say that two dividing partitions $\mathcal{P} = P_1 \sqcup \cdots \sqcup P_n$ and $\mathcal{Q} = Q_1 \sqcup \cdots \sqcup Q_n$ of $X$ \emph{differ by a $\beta$-shift} if there is a clopen subspace $S$ homeomorphic to $X_{\beta,1}$ such that, after possibly swapping the roles of $\cP$ and $\cQ$ and reindexing, we have \[
P_i = Q_i \text{ for } 1 < i < n, \quad P_1 = Q_1 \sqcup S \quad\text{ and }\quad P_n \sqcup S = Q_n.
\]
\end{DEF}

In other words, the subspace $S$ shifts ``out of'' $P_1$ and ``into'' $P_n$ to create the new partition $\mathcal{Q}$. 
When $\alpha=\beta +1$ is a successor ordinal, we call a $\beta$-shift a \emph{maximal shift}. We use the language of shifting because in \cite{BDHL2025} we show that there is always an infinite order homeomorphism of $X_{\alpha, n}$ that takes $\mathcal{P}$ to $\mathcal{Q}$. 

\begin{DEF}\label{def:TheGraph}
    Fix a successor ordinal $\alpha = \beta + 1$ and integer $n>1$. Define the graph $\Gamma(\alpha,n)$ to be the simplicial graph whose vertex set is the collection of dividing partitions of $X_{\alpha,n}$, where two partitions are connected by an edge when they differ by a $\beta$-shift.
\end{DEF}
In \cite[Section 4.3]{BDHL2025} we showed that the graph $\Gamma(\alpha,n)$ whose vertices are dividing partitions and where edges correspond to maximal shifts, is a Cayley--Abels--Rosendal graph for $\Homeo(X_{\alpha,n})$. 

\section{The Cantor--Bendixson derivative}\label{sec:derivative}
The following lemma is implicit in the proof of \cite[Lemma 17]{BDHL2025}.

\begin{lem}\label{lem:boundedorbits}
    Suppose $\mathcal{P} = P_1 \sqcup \cdots \sqcup P_n$ and $\mathcal{Q} = Q_1 \sqcup \cdots \sqcup Q_n$ are dividing partitions of $X_{\alpha, n}$. Fix $\beta < \alpha$.
    If for all $i$, each rank-$\beta$ point of $P_i$ belongs to $Q_i$ and vice versa,
    then $\mathcal{P}$ and $\mathcal{Q}$ differ by at most $n(n-1)$ $\beta$-shifts.
\end{lem}

\begin{proof}
    We will construct a path of $\beta$-shifts from $\mathcal{P}$ to $\mathcal{Q}$ of length at most $n(n-1)$.

    Choose a pair of distinct indices $(i,j)$ in $\{1,\ldots,n\}$.
    We will construct a new partition 
    \[ \mathcal{R} = R_1 \sqcup \cdots \sqcup R_n \]
    such that \begin{enumerate}
        \item $R_k = P_k$ if $k \neq i,j$, 
        \item $R_i \cap (P_i \sqcup P_j) \cap (Q_i \sqcup Q_j) = Q_i \cap (P_i \sqcup P_j)$,
        \item and therefore $R_j \cap (P_i \sqcup P_j) \cap (Q_i \sqcup Q_j) = Q_j \cap (P_i \sqcup P_j)$. 
    \end{enumerate}

    In other words, points of $P_j$ which move to $Q_i$
    and points of $P_i$ which move to $Q_j$
    will belong to $R_i$ and $R_j$, respectively,
    while points common to $P_i$ and $Q_i$ remain in $R_i$ and similarly for $R_j$.
    
    The partition $\mathcal{R}$ will by construction differ from $\mathcal{P}$ by at most two $\beta$-shifts.
    Repeating this construction for the pair $\mathcal{R},\mathcal{Q}$ and a different unordered pair of indices will produce $\mathcal{Q}$ after at most $\frac{n(n-1)}{2}$ repetitions, whence the claim.

    Here is the construction: choose a rank-$\beta$ point $x$ arbitrarily in $P_i \cap Q_i$ and a clopen neighborhood $Y$ of $x$ in $P_i \cap Q_i$ homeomorphic to $X_{\beta,1}$. By the classification of countable Stone spaces,
    because both $P_i \cap Q_j$ and $P_j \cap Q_i$ contain only points of rank strictly below $\beta$, both $Y_1 = Y \sqcup (Q_j \cap P_i)$ and $Y_2 = Y \sqcup (Q_i \cap P_j)$ are homeomorphic to $X_{\beta,1}$.
    The partition $\mathcal{S} = S_1 \sqcup \cdots \sqcup S_n$
    with \[
    S_k = P_k \text{ if } k \ne i,j, \quad S_i = P_i \setminus Y_1, \quad \text{ and } \quad S_j = P_j \sqcup Y_1
    \]
    therefore differs from $\mathcal{P}$ by a $\beta$-shift.
    Similarly the partition $\mathcal{R}$ with
    \[
    R_k = P_k \text{ if } k \ne i,j, \quad
    R_i = S_i \sqcup Y_2, \quad \text{ and } R_j = S_j \setminus Y_2
    \]
    differs from $\cS$ by a $\beta$-shift and has the desired properties above.
    This completes the claim and with it the proof.
\end{proof}

The lemma has the following corollary.

\begin{cor}\label{boundedorbitsobservation}
    Any subset $A \subset \Homeo(X_{\alpha,n})$ that acts trivially on the set of rank-$\beta$ points of $X_{\alpha,n}$ for $\beta < \alpha$ is coarsely bounded.
\end{cor}

\begin{proof}
    By \Cref{lem:boundedorbits}, if $\mathcal{P}$ is a dividing partition and $f \in \Homeo(X_{\alpha,n})$ is a $\beta$-shift on $\mathcal{P}$, the set $A$ is contained in the coarsely bounded set $(\{f^{\pm1}\}H)^{n(n-1)}$.
\end{proof}

Now we have two more immediate corollaries that will be utilized in the successor ordinal case. 

\begin{cor}\label{cor:kerCB}
    The kernel of the action of $\Homeo(X_{\beta+1,n})$ on the closed subset $X_{\beta+1,n}^{(\beta)} \cong X_{1,n}$ is coarsely bounded in $\Homeo(X_{\beta+1,n})$.
\end{cor}

\begin{proof}
    An element $f$ is in the kernel of the action of $\Homeo(X_{\beta+1,n})$ on $X_{\beta+1,n}^{(\beta)}$ if and only if it fixes every point of rank $\beta$.
\end{proof}

\begin{cor}\label{cor:CantorBendixson-QI}
    For any countable ordinal $\beta$, the group $\Homeo(X_{\beta+1,n})$ is quasi-isometric to $\Homeo(X_{1,n})$.
\end{cor}

\begin{proof}
    The action of $\Homeo(X_{\beta + 1,n})$
    on the closed subset $X_{\beta+1,n}^{(\beta)} \cong X_{1,n}$
    exhibits a continuous, surjective homomorphism
    $\pi\colon \Homeo(X_{\beta+1,n}) \to \Homeo(X_{1,n})$
    with coarsely bounded kernel. Thus this map is a quasi-isometry by \Cref{prop:sesce}.
\end{proof}

\section{The successor ordinal case}\label{sec:successcase}

The purpose of this section is to prove the second case of \Cref{thm:mainclassification}, the case where $\alpha = \beta + 1$ is a successor ordinal. 

\subsection{Cayley--Abels--Rosendal graphs as Hamming graphs}\label{sec:CARtoHam}

In this subsection, we prove that the Cayley--Abels--Rosendal graphs given in \Cref{subsec:CARgraphs} are isomorphic to Hamming graphs. 

Let $\mathbf{x}$ and $\mathbf{y}$ be sequences over some alphabet $A$. The \emph{Hamming distance} between $\mathbf{x}$ and $\mathbf{y}$
records the number of entries where the sequences differ.
When our alphabet is $\mathbb{Z}/n\mathbb{Z}$,
recall that the direct sum $\bigoplus_\omega \mathbb{Z}/n\mathbb{Z}$
may be identified with the subspace of the direct product $\prod_\omega \mathbb{Z}/n\mathbb{Z}$
comprising those sequences $(x_k)_{k\in\omega}$
for which the set $\{ k : x_k \ne 0 \}$ is finite.
Any two sequences $\mathbf{x}$ and $\mathbf{y}$ in $\bigoplus_\omega\mathbb{Z}/n\mathbb{Z}$ are thus at finite Hamming distance.

Given $h \ne 0$ in $\mathbb{Z}/n\mathbb{Z}$ and $k \in \omega$,
let $h\cdot e_k$ denote the sequence whose $k$th coordinate is $h$
and all others are $0$.

\begin{DEF}
The \emph{infinite Hamming graph} $H_n$ is the Cayley graph
of $\bigoplus_\omega \mathbb{Z}/n\mathbb{Z}$
with respect to the infinite generating set
$S = \{ h\cdot e_k : k \in \omega,~ h \in \mathbb{Z}/n\mathbb{Z} \setminus\{0\} \}$.
\end{DEF}

The induced graph metric on the vertices of $H_n$ agrees with the Hamming distance on $\bigoplus_\omega\mathbb{Z}/n\mathbb{Z}$. 

\begin{prop}\label{prop:CARisHamming}
    The graphs $\Gamma_n = \Gamma(1,n)$ and $H_n$ are isomorphic.
\end{prop}

First we need a short lemma. 

\begin{lem}\label{lem:partitionsarefunctions}
    Fix a labeling of the maximal points of $X_{\alpha,n}$ by elements of $\mathbb{Z}/n\mathbb{Z}$.
    Given a dividing partition $\mathcal{Q}$ of $X_{\alpha,n}$, label the partition elements with elements of $\mathbb{Z}/n\mathbb{Z}$ so that the $i$th maximal point belongs to $Q_i$ with indices mod $n$.
    This done, there is a bijective correspondence between dividing partitions and continuous functions $X_{\alpha,n} \to \mathbb{Z}/n\mathbb{Z}$ such that the $i$th maximal point is mapped to $i \in \mathbb{Z}/n\mathbb{Z}$.
\end{lem}

\begin{proof}
    Giving $\mathbb{Z}/n\mathbb{Z}$ the discrete topology, a function $\mathcal{Q} \colon X_{\alpha,n} \to \mathbb{Z}/n\mathbb{Z}$ is continuous when each preimage $Q_i = \mathcal{Q}^{-1}(i)$ is clopen. If the $i$th maximal point belongs to $Q_i$, this is a dividing partition.

    Conversely, given two dividing partitions $\mathcal{P}$ and $\mathcal{Q}$, after relabeling so that the $i$th maximal point belongs to $P_i$ and $Q_i$, we see that the dividing partitions are equal if and only if the corresponding functions are identical.
\end{proof}

\begin{proof}[Proof of \Cref{prop:CARisHamming}]
    We define a bijection $\Phi$ from vertices of $\Gamma_n$, that is dividing partitions of $X_{1,n}$,
    to elements of $\bigoplus_\omega \mathbb{Z}/n\mathbb{Z} = VH_n$
    which we will show is a graph isomorphism.
    
    First, fix a dividing partition $\mathcal{P}$ to serve as a basepoint, and enumerate the nonmaximal points of $X_{1,n}$ as $\{x_k\}_{k\in\omega}$.
    Define
    \[ \Phi(\mathcal{Q}) = (\mathcal{Q}(x_k) - \mathcal{P}(x_k))_{k \in \omega}. \]
    Here the subtraction is done in the group $\mathbb{Z}/n\mathbb{Z}$ and we use \Cref{lem:partitionsarefunctions} to identify $\mathcal{P}$ and $\mathcal{Q}$ with their corresponding continuous functions $X_{1,n} \to \mathbb{Z}/n\mathbb{Z}$.
    By the pigeonhole principle, if $\Phi(\mathcal{Q})$ has infinitely many nonzero terms, infinitely many of them come from some $P_i$,
    contradicting the fact that (when canonically labeled), the functions
    $\mathcal{Q}$ and $\mathcal{P}$ are continuous and agree on the $i$th maximal point.
    The function $\Phi$ is thus well-defined.

    An inverse to $\Phi$ takes an element $(z_k)_{k\in\omega}$ of $\bigoplus_\omega\mathbb{Z}/n\mathbb{Z}$
    to the function $\mathcal{Q}\colon X_{1,n} \to \mathbb{Z}/n\mathbb{Z}$
    which agrees with $\mathcal{P}$ on the maximal points
    and assigns the nonmaximal point $x_k \in X_{1,n}$ to
    the partition element with index
    \[
    \mathcal{Q}(x_k) = \mathcal{P}(x_k) + z_k
    \]
    computed in $\mathbb{Z}/n\mathbb{Z}$.
    Since $\mathcal{Q}$ differs from $\mathcal{P}$
    on finitely many isolated points,
    $\mathcal{Q}$ is a dividing partition.
    These operations are mutual inverses,
     so $\Phi$ is a bijection.

    Two vertices
    $\Phi(\mathcal{Q})$ and $\Phi(\mathcal{Q}')$ are adjacent in $H_n$
    when their difference $\Phi(\mathcal{Q}) - \Phi(\mathcal{Q}')$
    has one nonzero entry.
    This nonzero entry corresponds to one isolated point $x$
    that moves from $\mathcal{Q}(x)$ to $\mathcal{Q}'(x)$.
    The partitions $\mathcal{Q}$ and $\mathcal{Q}'$ thus differ by a $0$-shift.
    Conversely, two partitions $\mathcal{Q}$ and $\mathcal{Q}'$ that
    differ by a $0$-shift have the property that $\Phi(\mathcal{Q}) - \Phi(\mathcal{Q}')$
    has exactly one nonzero entry.
    Thus the map $\Phi$ is a graph isomorphism.
\end{proof}

The previous proposition and \Cref{cor:CantorBendixson-QI}
have the following corollary.

\begin{cor}
    If $\alpha$ is a countable successor ordinal and $n > 1$,
    $H_n$ is a Cayley--Abels--Rosendal graph for $\Homeo(X_{\alpha,n})$.
    \qed
\end{cor}

\subsection{Bi-Lipschitz Hamming graphs}\label{sec:HammingCubes}

The purpose of this section is to prove that all of the graphs
$H_n$ for $n > 1$ are bi-Lipschitz equivalent. The proof uses a Cantor--Schroeder--Bernstein type result for locally infinite graphs and the following ``swindle''.

\begin{lem}\label{lem:sumsofHamming}
    The spaces \[
    \bigoplus_\omega \mathbb{Z}/n\mathbb{Z} \text{ and } \bigoplus_\omega\left(\bigoplus_\omega\mathbb{Z}/n\mathbb{Z}\right)
    \]
    are isometric
    when $\bigoplus_\omega \mathbb{Z}/n\mathbb{Z}$ is equipped with the Hamming distance and $\bigoplus_\omega\left(\bigoplus_\omega\mathbb{Z}/n\mathbb{Z}\right)$
    is equipped with the $\ell^1$ sum of Hamming distances.
\end{lem}

\begin{proof}
    The isometry comes from partitioning the countably infinite index set
    into countably infinitely many copies of itself.
\end{proof}

If we consider Cayley graphs, equipping $\Z/n\Z$ with the complete generating set and each copy of $\bigoplus_\omega \mathbb{Z}/n\mathbb{Z}$ in the infinite direct sum with the generating set 
$S = \{ h\cdot e_k : k \in \mathbb{N}, h \in \mathbb{Z}/n\mathbb{Z} \}$,
the isometry above yields an isomorphism of Cayley graphs.
Noting that $\bigoplus_\omega \mathbb{Z}/n\mathbb{Z}=VH_n$, the previous lemma motivates the following definition of a direct sum of graphs.

\begin{DEF}
    Given a sequence of basepointed simplicial graphs $(\Gamma_i,\star_i)$ over an ordered set $I$, the simplicial graph $\bigoplus_{i\in I} \Gamma_i$ is the graph whose vertex set is the subset of $\prod_{i\in I} V\Gamma_i$ comprising those sequences $(v_i)_{i\in I}$ for which all but finitely many of the vertices $v_i$ satisfy $v_i = \star_i$, where $(v_i)$ is adjacent to $(w_i)$ if all but one of the coordinates are equal and the remaining pair of vertices $v_j \ne w_j$ are adjacent in $\Gamma_j$.
\end{DEF}

With this language and basepoint $\mathbf{0}$, \Cref{lem:sumsofHamming} shows that $H_n$ and $\bigoplus_\omega H_n$, are isomorphic. In general, the direct sum of Hamming graphs will again be a Hamming graph. By abuse of notation we will no longer distinguish between the graph $H_n$ and its vertex set and will always use $\mathbf{0}$ as a basepoint. 

Let $K_m$ denote the complete graph on $m$ vertices;
it is the Cayley graph of $\mathbb{Z}/m\mathbb{Z}$
with respect to the generating set $\{ h \in \mathbb{Z}/m\mathbb{Z} : h \ne 0 \}$.

\begin{lem}\label{lem:absorption}
    For any $n,m \geq 2$, there is a $4$-bi-Lipschitz bijection between the vertex sets of $H_n$ and $K_m \oplus H_n$.
\end{lem}

The proof is a variant of the classic ``back-and-forth'' argument
used to prove the Cantor--Schroeder--Bernstein theorem. Let us remark that for our application below, we really need a \emph{bijection}. A general nonsurjective quasi-isometry, when we swindle, would leave an unbounded gap in the image of the product map.

\begin{proof}
    Let $\iota\colon H_n \hookrightarrow K_m\oplus H_n$ denote the inclusion map $\iota(x)=(0,x)$; it is an isometric embedding. 
    Let $f\colon K_m\oplus H_n \arr H_n$ be the forgetful map $f(\ell,x)=x$; it is $1$-Lipschitz.
    Observe that $f\circ\iota$ is the identity,
    while $\iota \circ f$ is $1$-Lipschitz.

    Begin by enumerating the vertices of $H_n$ as $\{x_i\}_{i\in \N}$ 
    and the vertices of $K_m \oplus H_n$ as $\{y_i\}_{i\in \N}$. Next we define a bijection $\psi:H_n \arr K_m \oplus H_n$ by declaring $\psi(\mathbf{0})=(0,\mathbf{0})$ and using the following ``back and forth" argument. Repeat the following two steps:

    \begin{enumerate}
        \item Let $x_i \in H_n$ be the vertex with smallest index for which $\psi(x_i)$ is not yet defined.
        If $\iota(x_i) = (0,x_i)$ is not yet equal to $\psi(x_j)$ for $j < i$, define $\psi(x_i) = \iota(x_i)$.

        Otherwise, there are infinitely many neighbors of $\iota(x_i)$, so we may choose one, say $y$, that is not yet equal to $\psi(x_j)$ for $j < i$ and define $\psi(x_i) = y$.

        \item Now let $y_i \in K_m \oplus H_n$ be the vertex with smallest index that is not already equal to $\psi(x_j)$ for some $j \le i$.
        If $f(y_i) \ne x_j$ for any $j \le i$, define $\psi(f(y_i)) = y_i$.

        Otherwise, there are infinitely many neighbors of $f(y_i)$, so we may choose one, say $x$, on which $\psi$ is not yet defined and define $\psi(x) = y_i$.
    \end{enumerate}

    Continue in this manner, alternately assigning the image of $x_i$'s and $y_i$'s to be within one of their image under the inclusion or forgetful map respectively.
    The map $\psi$ constructed by induction is a bijection.

    To see that $\psi$ is Lipschitz, it is enough to check what happens to adjacent vertices, say $x$ and $x'$. 
    To fix notation, let $\psi(x) = (\ell,z)$ and $\psi(x') = (\ell',z')$. By construction, $x$ and $z$ (respectively $x'$ and $z'$) are either equal or adjacent in $H_n$. Thus we list the vertices of a path in $K_m \oplus H_n$ between $\psi(x)$ and $\psi(x')$ as
    \[
    \psi(x) = (\ell,z), (\ell',z), (\ell',x), (\ell',x'), (\ell',z') =  \psi(x').
    \]
    This path has length at most $4$, so $\psi$ is $4$-Lipschitz, as desired.
    
    Similarly, if $y$ and $y'$ are adjacent in $K_m \oplus H_n$,
    we have a path
    \[
    \psi^{-1}(y), f(y), f(y'), \psi^{-1}(y')
    \]
    in $H_n$ of length at most $3$, so we see that $\psi^{-1}$
    is $3$-Lipschitz.
\end{proof}

\begin{thm}\label{thm:BLequiv}
    For each $n,m \geq 2$, the spaces $H_n$ and $H_m$ are $16$-bi-Lipschitz equivalent.
\end{thm}
    \begin{proof} First observe that the product map \begin{align*}
        \Psi\colon \bigoplus_\omega H_n &\to \bigoplus_\omega (K_m \oplus H_n) \\
        (z_i) &\mapsto (\psi(z_i))
    \end{align*} is a $4$-bi-Lipschitz bijection.
    
    Now we have the following isomorphisms of graphs (or of spaces where each sum is equipped with the $\ell^1$ metric). 
    \begin{align*}
    \bigoplus_\omega (K_m \oplus H_n)
    \cong
    \left(\bigoplus_\omega K_m\right)\oplus
    \left(\bigoplus_\omega H_n\right)
    \cong
        H_m\oplus H_n.
    \end{align*} 
    So $H_n$ is $4$-bi-Lipschitz equivalent to $H_m \oplus H_n$.
    By symmetry, $H_m$ is also $4$-bi-Lipschitz equivalent to $H_m \oplus H_n$. 
    Thus, $H_m$ and $H_n$ are $16$-bi-Lipschitz equivalent.
\end{proof}

\subsection{Assembling the pieces}

We are now ready to prove the successor case of \Cref{thm:mainclassification}.

\begin{thm}
    For any countable successor ordinal $\alpha$ and any $n > 1$,
    the groups $\Homeo(X_{\alpha,n})$ and $\Homeo(X_{1,2})$ are quasi-isometric. Furthermore, they are all quasi-isometric to $H_2$. 
\end{thm}

\begin{proof}
    By~\Cref{cor:CantorBendixson-QI}, the group
    $\Homeo(X_{\alpha,n})$ is quasi-isometric to $\Homeo(X_{1,n})$.
    By \Cref{prop:CARisHamming},
    $H_n$ is a Cayley--Abels--Rosendal graph for $\Homeo(X_{1,n})$.
    Therefore we have that $H_n$, $\Homeo(X_{1,n})$ and $\Homeo(X_{\alpha,n})$
    are all quasi-isometric.
    By \Cref{thm:BLequiv}, $H_n$ is $16$-bi-Lipschitz equivalent
    to $H_2$, which is a Cayley--Abels--Rosendal graph for $\Homeo(X_{1,2})$.
\end{proof}

\section{The limit ordinal case}\label{sec:limitcase}

The main purpose of this section
is to prove the limit ordinal case of \Cref{thm:mainclassification}.
Along the way, we prove \Cref{thm:maintrees}
and an extension, \Cref{prop:orbit-map-CE}, to topological groups $G$ countably but not finitely generated over a proper open subgroup $H$ coarsely bounded in $G$.

We consider groups $G$
with a chain $G_1 \le G_2 \le \cdots$ of subgroups.
In this situation,
the following classical result of Bass--Serre theory applies.

\begin{lem}[proof of Theorem 15 \cite{Trees}]\label{lem:serretrees}
    Suppose that $G_1 \le G_2 \le \cdots$ is a chain of subgroups of a group $G$.
    The disjoint union of the sets of cosets $G / G_k$ organize into the vertices of a graph on which $G$ acts, 
    where each coset $g G_k$ is connected by an edge to the coset $g G_{k+1}$. When $\bigcup_{k=1}^\infty G_k = G$, the graph is a tree $T$ (in particular it is connected).

    When $G$ is a topological group, the action on $T$ is continuous if and only if each subgroup $G_k$ is open.
    When each $G_k$ is a proper subgroup of $G$, the action on $T$ is without global fixed point.
\end{lem}

The quotient graph of groups $G \backslash T$ for $T$ the tree obtained via \Cref{lem:serretrees} is a ray with each vertex group $G_k$ equal to the outgoing edge group and including into $G_{k+1}$. The situation is pictured in \Cref{fig:bstree2}.

\subsection{Coarse equivalences} \label{ssec:generaltrees}

The purpose of this subsection is to answer the following questions of a group $G$ admitting an exhaustion
$G_1 \le G_2 \le \cdots$ by proper, open subgroups.
To what extent can $T$ (or $G$-orbits in $T$) be used to construct a geometric model for the coarse structure of $G$?
Can an orbit in $T$ itself ever be used?

Similar questions are addressed in~\cite{KopreskiShaji}. 
There, the authors consider vertex-transitive \emph{metric} graphs with edge lengths tending to infinity.
Our graphs are not vertex-transitive,
but we do not vary edge lengths.
Our proofs of theorems in this subsection use tools from \cite{KopreskiShaji}, but our proofs do not quite follow directly from their work.

\begin{figure}[ht!]\centering
        \def\svgwidth{\textwidth}
        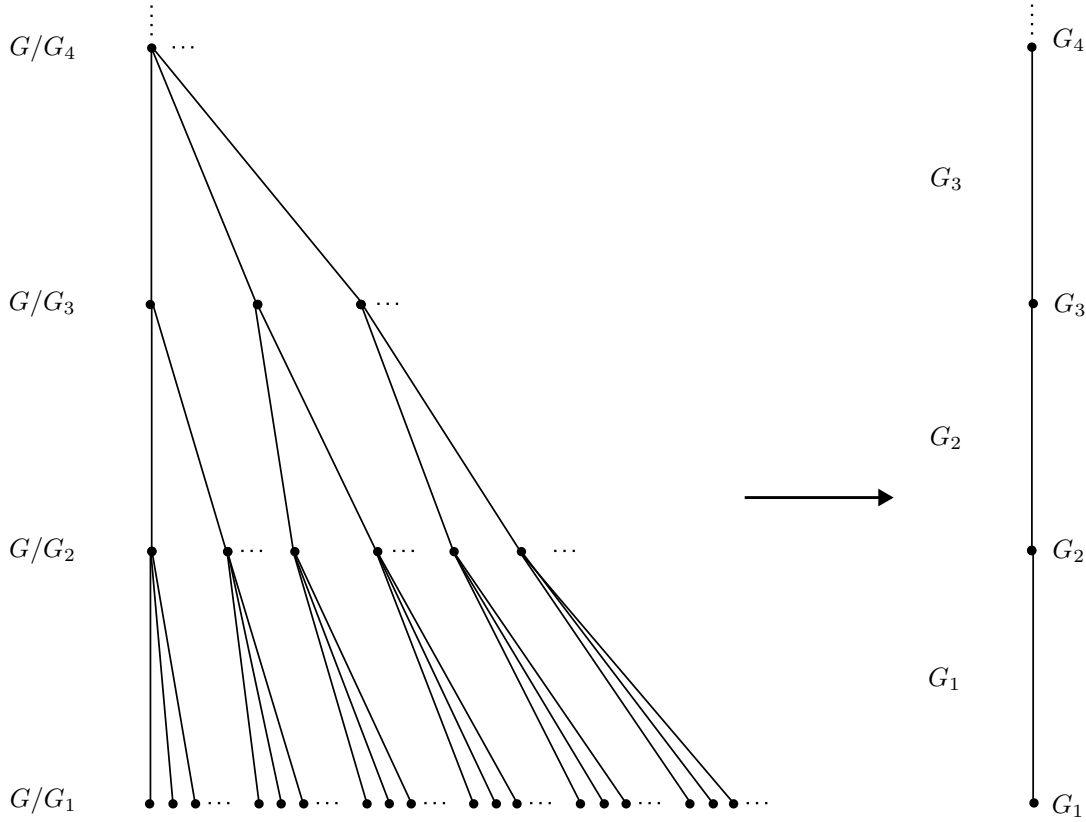
		\caption{A schematic of the Bass--Serre $T$ together with the quotient ray of groups $G \backslash T$ as in \Cref{lem:serretrees}. In the tree $T$, all of the bottommost vertices are indexed via cosets of $G_{1}$, then the next are indexed via cosets of $G_2$, etc. }
        \label{fig:bstree2}
\end{figure} 

The following restates \Cref{thm:maintrees}.

\begin{thm}\label{thm:CEtoleaves}
    Suppose that $G = \bigcup_{k=1}^\infty G_k$, where each $G_k$ is a proper, open subgroup coarsely bounded in $G$. Let $T$ be the tree obtained from \Cref{lem:serretrees}. Taking the vertex $v = G_1$ in $T$, the orbit map $g \mapsto g.v$ from the action of $G$ on $T$ is a coarse equivalence between $G$ with its canonical coarse structure and the set of leaves in $T$ equipped with the subspace metric.
\end{thm}

\begin{proof}
    We work as in the proof of \cite[Proposition 4.3]{KopreskiShaji}.
    To be a coarse equivalence, the orbit map must \emph{a priori} satisfy three properties: bornologous, expanding and cobounded.
    By \cite[Lemma 2.10]{KopreskiShaji}, all isometric and continuous actions have bornologous orbit maps.
    Because the orbit map is surjective onto $G.v$, it is cobounded.
    By \cite[Lemma 2.11]{KopreskiShaji}, showing the orbit map is expanding is equivalent to showing that preimages of bounded subsets in $G.v$ are coarsely bounded in $G$. 
    We show this property holds to complete the proof.

    Let $B \subset G.v$ be bounded, and let $A \subset G$ be its full preimage. 
    Since $B$ is bounded, we have a uniform bound, say $M$, so that $d_T(v,a.v) \leq M$ for all $a\in A$. 
    We thus have that any path from $v$ to $a.v$ is contained in a single connected component of the subgraph $T_{\leq L}$ of $T$ spanned by vertices of height at most $L = \frac{M}{2}$.
    This connected component has a single vertex of maximal height $L$ corresponding to the coset $G_L$. 
    Thus $A$ stabilizes this vertex and is contained in the subgroup $G_L$. 
    As $G_L$ is coarsely bounded in $G$, so is $A$. 
\end{proof}

The above theorem partially answers our second question.
The reader interested only in \Cref{thm:mainclassification} may skip ahead to the next subsection.
To answer the question in the case that not all $G_k$ are coarsely bounded in $G$, we turn to the following setting.

Suppose now that $G$ has an open, coarsely bounded subgroup $H$. If $G$ is generated by $H$ together with a countably infinite set $\{s_n : n \in \mathbb{N}\}$,
we obtain an exhaustion $G_1 \le G_2 \le \cdots$,
where $G_k = \langle H, s_1,\ldots,s_k\rangle$,
so that \Cref{lem:serretrees} applies.
If some (and hence all sufficiently large) $G_k$
is \emph{not} coarsely bounded in $G$,
then we cannot conclude that orbit maps to $T$
are coarse equivalences.
Instead we will use the tree $T$ to construct a new graph $\Gamma$ equipped with a $G$-action, modeled on a tree-of-spaces construction \cite[Chapter II.11]{BridsonHaefliger}.

\begin{DEF}[Tree of Cayley--Abels--Rosendal Graphs]\label{DEF:TreeofCB}
    We continue with the notation $G$, $H$, $S$ and $T$ as above, and suppose that each $G_k$ is a proper subgroup of $G$.

    We now construct the graph $\Gamma$.
    Let $\mathbb{R}_+$ denote the graph with vertex set $\mathbb{N}$ and an oriented edge from $k$ to $k+1$ for all $k$.
    We begin with the graph $G/H \times \mathbb{R}_+$;
    the graph $\Gamma$ will be formed by adding more edges.
    We call edges already present in $G/H \times \mathbb{R}_+$ \emph{vertical},
    and edges which we will add \emph{horizontal,}
    since newly added edges will always connect
    vertices $(g H, k)$ whose second coordinates are equal.
    The group $G$ will thus act continuously on this graph by multiplication in the labels.

    The horizontal edges are defined similarly to those in \cite{BDHL2025}: we connect the vertex $(fH, k)$ to all vertices of the form $(fhs_i H, k)$,
    where $h \in H$, $s_i \in S$, and $i \le k$.
    We call the resulting graph $\Gamma$
    a \emph{tree of Cayley--Abels--Rosendal graphs} for $G$.
\end{DEF}
    
We record a number of features of this graph $\Gamma$.

Firstly, the graph is countable. This was a hypothesis for Cayley--Abels--Rosendal graphs in \cite{BDHL2025}, but here it is a conclusion of $G$ being countably generated over $H$.
    
Two $H$-cosets of the form $fH$ and $fhs_iH$ for $h \in H$ and $1\leq i \leq k$ must lie in the same $G_k$-coset. One can picture $\Gamma$ by taking the Bass--Serre tree in \Cref{fig:bstree2} and blowing up a vertex corresponding to a coset of $G_k$ in $G$,
replacing it with a connected graph whose vertex set is the countable set of cosets $G_k / H$.

The action of $G$ on $\Gamma$ is height-preserving and so has countably infinitely many orbits of vertices. Each vertex stabilizer is conjugate to the open subgroup $H$, so the action is continuous. 
    
Each horizontal edge corresponds to an element $s_i \in S$. Each vertex at height $k$ has a unique vertical edge connecting it to a vertex at height $k+1$, and when $k>1$, a unique vertical edge connecting to a vertex at height $k-1$. There is a height-preserving graph map $\Gamma \rightarrow T$ defined by collapsing all vertices which are connected by a horizontal path; i.e.\ all $H$-cosets which belong to the same $G_k$-coset.

The following is our generalization of \Cref{thm:CEtoleaves}.
    
\begin{prop}\label{prop:orbit-map-CE}
    With notation as above, 
    if $v \in \Gamma$ is a vertex,
    the orbit map $g \mapsto g.v$ from the action of $G$ on $\Gamma$ is a coarse equivalence between $G$ with its canonical coarse structure
    and $G.v$ equipped with the subspace metric.
\end{prop}

\begin{proof}
    We again work as in the proof of \cite[Proposition 4.3]{KopreskiShaji} and \Cref{thm:CEtoleaves}. It suffices to show that the preimages of bounded subsets in $G.v$ are coarsely bounded in $G$. 

    Let $B \subset G.v$ be bounded, and let $A \subset G$ be its full preimage. 
    Since $B$ is bounded, we have a uniform bound, say $M$, so that $d_\Gamma(v,a.v) \leq M$ for all $a\in A$. 
    Assuming that the height of the vertex $v$ is $N$, we have that any path from $v$ to $a.v$ is contained in the subgraph, $\Lambda_{\leq L}$ of $\Gamma$ spanned by vertices of height at most $L = N + \frac{M}{2}$. 
    Each edge in $\Lambda_{\leq L}$ is either vertical, or a horizontal edge labeled with an element of the set $\{s_1,\ldots,s_L\} \subset S$. 
    Thus by taking $F = \{s_1,\ldots,s_L\}$ we have that $a \in (FH)^{M+1}$. 
    Since $H$ is coarsely bounded in $G$ and $F$ is finite, the set $(FH)^{M+1}$ is coarsely bounded in $G$ and so $A$ is also. 
\end{proof}

The following existence result provides a wealth of cases in which \Cref{DEF:TreeofCB} and \Cref{prop:orbit-map-CE} may be applied.

\begin{thm}
    Suppose that $G$ is a locally bounded, non-Archimedean Polish group. If $G$ is not generated by a CB set, there exists an open, coarsely bounded subgroup $H \le G$ and a countably infinite set $S = \{s_k\}$ satisfying the conditions in \Cref{DEF:TreeofCB} and \Cref{prop:orbit-map-CE}. Therefore, there is a tree of Cayley--Abels--Rosendal graphs $\Gamma$ for $G$.
\end{thm}

\begin{proof}
    Since $G$ is locally bounded, there is a coarsely bounded identity neighborhood $U \subset G$. Since $G$ is non-Archimedean, $U$ contains an open subgroup $H$ that is also coarsely bounded. Since $G$ is Polish, it is countably generated over $H$, say by some countable set $S = \{s_n : n \in \mathbb{N}\}$. The subgroup $H$ together with any finite subset of $S$ is a coarsely bounded set, so cannot generate $G$ by assumption. Therefore each subgroup $G_k = \langle H, s_1,\ldots,s_k\rangle $ is a proper subgroup of $G$, so that \Cref{DEF:TreeofCB} applies.
\end{proof}

\subsection{Limit ordinals}\label{ssec:limitfinish}

The purpose of this subsection is to prove the limit ordinal case of \Cref{thm:mainclassification}.
Suppose $\lambda = \lim_{k\to\infty} \alpha_k$ is a countable limit ordinal.

Fix a dividing partition $\mathcal{P}$ of $X_{\lambda,n}$.
The stabilizer of $\mathcal{P}$, call it $H$, is open and coarsely bounded in $\Homeo(X_{\lambda,n})$ by~\Cref{lem:stabCB}.
Given $k \in \mathbb{N}$, choose an $\alpha_k$-shift $s_k$ in $\Homeo(X_{\lambda,n})$.
Because $\lim_{k\to\infty}\alpha_k = \lambda$,
we have \[\bigcup_{k=1}^\infty G_k = \Homeo(X_{\lambda,n}),\]
where $G_k = \langle H, s_1,\ldots,s_k\rangle$.
However, each $G_k$ is a proper open subgroup of $\Homeo(X_{\lambda,n})$.

As we showed in \cite[Proposition 19]{BDHL2025},
these subgroups show that $\Homeo(X_{\lambda,n})$
is locally CB but not generated by any coarsely bounded set.
Now the limit case of \Cref{thm:mainclassification} will follow from an application of \Cref{thm:CEtoleaves} utilizing these subgroups.

\begin{lem}\label{lem:GkareCB}
    For each $k$, the subgroup $G_k$ is CB in $\Homeo(X_{\lambda,n})$.
\end{lem}

\begin{proof}
    Fix $k$ and let $g \in G_k$ be arbitrary. By the definition of $G_k$, $\cP$ and $g\cP$ do not differ on any rank-$\alpha_{k+1}$ points. This allows us to apply \Cref{lem:boundedorbits} in order to write $g$ as a product of at most $n(n-1)$ $\alpha_{k+1}$-shifts. Setting $F = \{\id,s_1,\ldots,s_{k+1}\}$ we thus have that $G_k$ is contained in the coarsely bounded set $(F H)^{n^2}$.
\end{proof}

\begin{thm}\label{thm:limitcase}
    Let $\lambda$ be a limit ordinal and $n>1$, then $\Homeo(X_{\lambda,n})$ is coarsely equivalent to the set of leaves of the regular one-ended tree of countably infinite valence. 
\end{thm}

\begin{proof}
    The realization of $\Homeo(X_{\lambda,n})$ as $\bigcup_{k=1}^{\infty} G_k$ together with \Cref{lem:GkareCB} allows us to apply \Cref{thm:CEtoleaves}. Thus we see that $\Homeo(X_{\lambda,n})$ is coarsely equivalent to the set of leaves of the Bass--Serre tree $T$ given in \Cref{lem:serretrees}. Since $\alpha_k$-shifts have infinite order, the index of $G_k$ in $G_{k+1}$ is infinite, so $T$ is the regular one-ended tree of countably infinite valence. 
\end{proof}

This finishes the proof of the limit case of \Cref{thm:mainclassification}. 

\section{Looking towards surfaces}

\label{sec:surfaces}

One motivation for studying homeomorphism groups of Stone spaces is to better understand mapping class groups of infinite-type surfaces. For an infinite-type surface $\Sigma$, its space of ends, $E(\Sigma)$, is a second countable Stone space. Furthermore, the mapping class group, $\MCG(\Sigma)$, surjects onto the homeomorphism group of its end space, giving the following short exact sequence. 
\begin{align*}
    1 \longrightarrow \PMCG(\Sigma) \longrightarrow \MCG(\Sigma) \longrightarrow \Homeo(E(\Sigma),E_g(\Sigma)) \longrightarrow 1
\end{align*}
where $E_{g}(\Sigma) \subset E(\Sigma)$ is the (closed) subspace corresponding to ends accumulated by genus (often referred to as non-planar ends) and $\PMCG(\Sigma)$ denotes the pure mapping class group of $\Sigma$. One may hope to use this short exact sequence in order to upgrade the above results to the setting of mapping class groups of infinite-type surfaces with countable end spaces. Unfortunately, as seen below, this does not work directly in most cases. 
That said, the techniques and results above may still offer a strategy for better understanding the coarse geometry of mapping class groups of infinite-type surfaces. In particular, given that we show that the coarse geometry of homeomorphism groups of countable Stone spaces degenerates into only three categories, if one hopes for any type of quasi-isometric rigidity among mapping class groups of infinite-type surfaces, one \emph{must} make use of the surface topology. 

We first recall a tool helpful for exhibiting the failure of a subset to be coarsely bounded. A \emph{continuous length function} on a topological group $G$ is a continuous map $\ell: G \rightarrow [0,\infty)$ such that $\ell(\id)=0$, $\ell(g) = \ell(g^{-1})$, and $\ell(gh) \leq \ell(g) + \ell(h)$ for all $g,h\in G$. If $\ell$ is a length function on $G$ such that $\ell$ is unbounded on a subset $H \subset G$, then $H$ is \emph{not} coarsely bounded in $G$. In particular, given a continuous length function $\ell$ on $G$, one can obtain a continuous left-invariant pseudometric by taking $d(g,h) = \ell(g^{-1}h)$ for all $g,h\in G$. 

Now we can utilize a length function defined in \cite[Section 2]{MR2023} to immediately show that the quotient map above is not a coarse equivalence for a large family of surfaces. We say that a subsurface $K \subset \Sigma$ is \emph{nondisplaceable} if $K$ and $f(K)$ intersect essentially for all $f \in \Homeo(\Sigma)$. 

\begin{lem}\label{lem:forgetfulsurface}
    If $\Sigma$ has a finite-type nondisplaceable subsurface, then $\PMCG(\Sigma)$ is not coarsely bounded in $\MCG(\Sigma)$. In particular, the quotient map $\MCG(\Sigma) \rightarrow \Homeo(E(\Sigma),E_g(\Sigma))$ is not a coarse equivalence. 
\end{lem}
\begin{proof}
    In \cite[Section 2]{MR2023}, given a finite-type nondisplaceable subsurface $K \subset \Sigma$, the authors construct a continuous length function $\ell$ on $\MCG(\Sigma)$ that is unbounded on the subgroup $\<f\>$ for any mapping class $f \in \MCG(\Sigma)$ that preserves $K$ and restricts to a pseudo-Anosov mapping class on $K$. Now, picking a finite-type subsurface $K$ in $\Sigma$ of sufficient complexity, there exists a pseudo-Anosov mapping class $g_0 \in \PMCG(K)$. By extending $g_0$ via the identity we obtain a pure mapping class $g \in \PMCG(\Sigma)$ for which $\ell$ is unbounded on $\<g\>$. 
\end{proof}

Next we use length functions to check that the puncture-forgetting ``Cantor--Bendixson derivative'' does not induce a coarse equivalence on mapping class groups of surfaces, as opposed to the case of countable Stone spaces. 

\begin{lem}\label{lem:canbendsurface}
    Let $\Sigma$ be a surface with at least one isolated planar end and a finite-type nondisplaceable subsurface and let $\Sigma'$ be the surface obtained from $\Sigma$ by forgetting all isolated planar ends. The kernel of the forgetful map $\mathcal{F}: \MCG(\Sigma) \rightarrow \MCG(\Sigma')$ is not coarsely bounded and hence $\mathcal{F}$ does not induce a coarse equivalence. 
\end{lem}

\begin{proof}
    We repeat the same proof as above, now using a finite-type nondisplaceable subsurface $K \subset \Sigma$ that contains at least one isolated planar end, $e$. Any map obtained by point pushing $e$ about a filling curve of $K$ is a pseudo-Anosov mapping class in $\PMCG(K)$ \cite{Kra1981} and hence can be extended via the identity to a map in $\ker(\mathcal{F})$ on which the Mann--Rafi length function is unbounded. 
\end{proof}

Both of these lemmas require the existence of a nondisplaceable subsurface. In \cite{BNQR2026}, the authors introduce a notion of ``complexity'' for stable infinite-type surfaces with CB-generated mapping class groups. Roughly speaking, the natural number $\zeta(\Sigma)$ measures the minimal complexity of a finite-type subsurface used to define a CB-generating set. We refer the reader to \cite[Section 2.5]{BNQR2026} for the formal definition. Notably, when $\Sigma$ has either zero or infinite genus, if $\zeta(\Sigma)=1$, then $\MCG(\Sigma)$ is itself CB, if $\zeta(\Sigma) = 2$, then $\Sigma$ does not contain a finite-type nondisplaceable subsurface, and if $\zeta(\Sigma)\geq 3$, then $\Sigma$ has a finite-type nondisplaceable subsurface. Furthermore, if $\zeta(\Sigma)=1$ then either $E_{g}(\Sigma) = E(\Sigma)$ or $E_g(\Sigma) = \emptyset$ and $E(\Sigma)$ is self-similar. Thus $\homeo(E(\Sigma))$ is also itself CB. This leaves one case in which we have not covered in the above lemmas, raising the question asked in the introduction, repeated below. 

\begin{customq}{1}
    Let $\Sigma$ be a stable surface with $\zeta(\Sigma)=2$ whose mapping class group is CB-generated but not coarsely bounded. Is $\PMCG(\Sigma)$ coarsely bounded in $\MCG(\Sigma)$? More specifically, let $\Sigma_{1,2}$ be the genus-zero surface with end space $X_{1,2}$. Is $\PMCG(\Sigma_{1,2})$ coarsely bounded inside of $\MCG(\Sigma_{1,2})$? 
\end{customq}

Surfaces $\Sigma$ as in the statement are exactly the ``translatable surfaces'' studied in \cite{SC2024} and ``avenue surfaces'' considered in \cite{HQR2022}. In \cite{SC2024} the author provides a Cayley--Abels--Rosendal graph for these mapping class groups. While we would be surprised if the answer to the above question were ``yes,'' the current tools used in \cite{MR2023} to certify that a mapping class group is not CB cannot be used to certify that these pure mapping class groups are not CB. In particular, in \cite[Section 7]{MR2023} the authors define an unbounded length function on $\MCG(\Sigma)$ that is actually bounded on $\PMCG(\Sigma)$. 

\bibliographystyle{alpha}
\bibliography{bib.bib}

\end{document}

%% file: preamble.tex
\usepackage{amsmath}
\usepackage{amsfonts}
\usepackage{bbold}
\usepackage[mathscr]{eucal}
\usepackage{amsthm}
\usepackage{amssymb}
\usepackage{hyperref}
\usepackage{cleveref}
\usepackage[margin=1.3in]{geometry}
\usepackage[dvipsnames]{xcolor}
\usepackage{graphicx}
\usepackage{mathtools}
\usepackage{multirow}
\usepackage[shortlabels]{enumitem}
\usepackage{import}
\usepackage{tikz-cd}
\usetikzlibrary{decorations.pathreplacing,calc}
\usepackage{soul}
\usepackage{accents}
\usepackage{tikz-cd}

\usepackage{amsthm}
\usepackage{thmtools}
\usepackage{thm-restate}

\usepackage{comment}

\newcounter{COMMENTS}

\makeatletter
\def\firsttwoletters#1#2\@nil{%
  \uppercase{#1}%
  \ifx\relax#2\relax
  \else
    \uppercase{\expandafter\@firstoftwo\expandafter{\expandafter\@car\string#2\@nil}}%
  \fi
}

\newcommand{\newComment}[3]{%
  \expandafter\newcommand\csname #1\endcsname[1]{
    \textbf{%
      \color{#3}(\firsttwoletters#2\@nil\theCOMMENTS)%
    }%
    \marginpar{\scriptsize\raggedright\textbf{%
      {\color{#3}(\expandafter\uppercase\expandafter{\firsttwoletters#2\@nil}\theCOMMENTS)#1: 
      }} ##1}%
    \stepcounter{COMMENTS}%
  }%
  \expandafter\newcommand\csname #2\endcsname[1]{{\color{#3} ##1}}
}
\makeatother

\newComment{Beth}{bb}{Plum}
\newComment{Robbie}{rl}{cyan}
\newComment{George}{gd}{orange}
\newComment{Hannah}{hh}{WildStrawberry}

\usepackage{tcolorbox} 
\newlist{todolist}{itemize}{2}
\setlist[todolist]{label=$\square$}
\usepackage{pifont}
\newtheorem{thm}{Theorem}[section]
\newtheorem{lem}[thm]{Lemma}
\newtheorem{cor}[thm]{Corollary}
\newtheorem{prop}[thm]{Proposition}

\theoremstyle{definition}
\newtheorem{DEF}[thm]{Definition}

\newtheorem{innercustomthm}{Theorem}
\newenvironment{customthm}[1]
  {\renewcommand\theinnercustomthm{#1}\innercustomthm}
  {\endinnercustomthm}

\newtheorem{innercustomcor}{Corollary}
\newenvironment{customcor}[1]
  {\renewcommand\theinnercustomcor{#1}\innercustomcor}
  {\endinnercustomcor}

\newtheorem{innercustomq}{Question}
\newenvironment{customq}[1]
  {\renewcommand\theinnercustomq{#1}\innercustomq}
  {\endinnercustomq}

\DeclareMathOperator{\homeo}{Homeo}

\DeclareMathOperator{\Homeo}{Homeo}

\DeclareMathOperator{\id}{Id}

\DeclareMathOperator{\MCG}{MCG}
\DeclareMathOperator{\PMCG}{PMCG}
\DeclareMathOperator{\Fix}{Fix}

\newcommand{\Z}{\mathbb{Z}}
\newcommand{\G}{\Gamma}
\newcommand{\N}{\mathbb{N}}

\newcommand{\cP}{\mathcal{P}}
\newcommand{\cQ}{\mathcal{Q}}

\newcommand{\cS}{\mathcal{S}}

\renewcommand{\>}{\rangle}

\newcommand{\arr}{\rightarrow}
\DeclareMathOperator{\fix}{Fix}

%% file: fig/Gamma2_partitions.pdf_tex
\begingroup%
  \makeatletter%
  \providecommand\color[2][]{%
    \errmessage{(Inkscape) Color is used for the text in Inkscape, but the package 'color.sty' is not loaded}%
    \renewcommand\color[2][]{}%
  }%
  \providecommand\transparent[1]{%
    \errmessage{(Inkscape) Transparency is used (non-zero) for the text in Inkscape, but the package 'transparent.sty' is not loaded}%
    \renewcommand\transparent[1]{}%
  }%
  \providecommand\rotatebox[2]{#2}%
  \newcommand*\fsize{\dimexpr\f@size pt\relax}%
  \newcommand*\lineheight[1]{\fontsize{\fsize}{#1\fsize}\selectfont}%
  \ifx\svgwidth\undefined%
    \setlength{\unitlength}{558.90643262bp}%
    \ifx\svgscale\undefined%
      \relax%
    \else%
      \setlength{\unitlength}{\unitlength * \real{\svgscale}}%
    \fi%
  \else%
    \setlength{\unitlength}{\svgwidth}%
  \fi%
  \global\let\svgwidth\undefined%
  \global\let\svgscale\undefined%
  \makeatother%
  \begin{picture}(1,0.22214623)%
    \lineheight{1}%
    \setlength\tabcolsep{0pt}%
    \put(0,0){\includegraphics[width=\unitlength,page=1]{Gamma2_partitions.pdf}}%
    \put(0.2873615,0.09763074){\color[rgb]{0,0,0}\makebox(0,0)[lt]{\lineheight{1.25}\smash{\begin{tabular}[t]{l}$0$\end{tabular}}}}%
    \put(0,0){\includegraphics[width=\unitlength,page=2]{Gamma2_partitions.pdf}}%
    \put(0.37511782,0.09763074){\color[rgb]{0,0,0}\makebox(0,0)[lt]{\lineheight{1.25}\smash{\begin{tabular}[t]{l}$1$\end{tabular}}}}%
    \put(0.45978406,0.09763074){\color[rgb]{0,0,0}\makebox(0,0)[lt]{\lineheight{1.25}\smash{\begin{tabular}[t]{l}$2$\end{tabular}}}}%
    \put(0.54383244,0.09763074){\color[rgb]{0,0,0}\makebox(0,0)[lt]{\lineheight{1.25}\smash{\begin{tabular}[t]{l}$3$\end{tabular}}}}%
    \put(0.20949327,0.09763074){\color[rgb]{0,0,0}\makebox(0,0)[lt]{\lineheight{1.25}\smash{\begin{tabular}[t]{l}$-1$\end{tabular}}}}%
    \put(0.12654178,0.09632822){\color[rgb]{0,0,0}\makebox(0,0)[lt]{\lineheight{1.25}\smash{\begin{tabular}[t]{l}$-\infty$\end{tabular}}}}%
    \put(0.65593781,0.0965127){\color[rgb]{0,0,0}\makebox(0,0)[lt]{\lineheight{1.25}\smash{\begin{tabular}[t]{l}$\infty$\end{tabular}}}}%
    \put(0,0){\includegraphics[width=\unitlength,page=3]{Gamma2_partitions.pdf}}%
    \put(0.87198465,0.20721995){\color[rgb]{0,0,0}\makebox(0,0)[lt]{\lineheight{1.25}\smash{\begin{tabular}[t]{l}$\mathcal{P}$\end{tabular}}}}%
    \put(0.78677362,0.11664128){\color[rgb]{0,0,0}\makebox(0,0)[lt]{\lineheight{1.25}\smash{\begin{tabular}[t]{l}$\mathcal{R}$\end{tabular}}}}%
    \put(0.35467985,0.16494991){\color[rgb]{0,0,0}\makebox(0,0)[lt]{\lineheight{1.25}\smash{\begin{tabular}[t]{l}$\mathcal{P}$\end{tabular}}}}%
    \put(0.26007548,0.14817608){\color[rgb]{0,0,0}\makebox(0,0)[lt]{\lineheight{1.25}\smash{\begin{tabular}[t]{l}$\mathcal{R}$\end{tabular}}}}%
  \end{picture}%
\endgroup%

%% file: fig/Gamma2_Zsubsets.pdf_tex
\begingroup%
  \makeatletter%
  \providecommand\color[2][]{%
    \errmessage{(Inkscape) Color is used for the text in Inkscape, but the package 'color.sty' is not loaded}%
    \renewcommand\color[2][]{}%
  }%
  \providecommand\transparent[1]{%
    \errmessage{(Inkscape) Transparency is used (non-zero) for the text in Inkscape, but the package 'transparent.sty' is not loaded}%
    \renewcommand\transparent[1]{}%
  }%
  \providecommand\rotatebox[2]{#2}%
  \newcommand*\fsize{\dimexpr\f@size pt\relax}%
  \newcommand*\lineheight[1]{\fontsize{\fsize}{#1\fsize}\selectfont}%
  \ifx\svgwidth\undefined%
    \setlength{\unitlength}{281.59816712bp}%
    \ifx\svgscale\undefined%
      \relax%
    \else%
      \setlength{\unitlength}{\unitlength * \real{\svgscale}}%
    \fi%
  \else%
    \setlength{\unitlength}{\svgwidth}%
  \fi%
  \global\let\svgwidth\undefined%
  \global\let\svgscale\undefined%
  \makeatother%
  \begin{picture}(1,0.66023114)%
    \lineheight{1}%
    \setlength\tabcolsep{0pt}%
    \put(0,0){\includegraphics[width=\unitlength,page=1]{Gamma2_Zsubsets.pdf}}%
    \put(0.25259592,0.63211853){\color[rgb]{0,0,0}\makebox(0,0)[lt]{\lineheight{1.25}\smash{\begin{tabular}[t]{l}$\mathcal{P}=\emptyset$\end{tabular}}}}%
    \put(0.22243551,0.43693064){\color[rgb]{0,0,0}\makebox(0,0)[lt]{\lineheight{1.25}\smash{\begin{tabular}[t]{l}$\mathcal{R} = \{0\}$\end{tabular}}}}%
    \put(0.52636425,0.45199693){\color[rgb]{0,0,0}\makebox(0,0)[lt]{\lineheight{1.25}\smash{\begin{tabular}[t]{l}$\{2\}$\end{tabular}}}}%
    \put(0.36974931,0.45131146){\color[rgb]{0,0,0}\makebox(0,0)[lt]{\lineheight{1.25}\smash{\begin{tabular}[t]{l}$\{1\}$\end{tabular}}}}%
    \put(0.42144503,0.19059662){\color[rgb]{0,0,0}\makebox(0,0)[lt]{\lineheight{1.25}\smash{\begin{tabular}[t]{l}$\{0,2\}$\end{tabular}}}}%
    \put(0.26963304,0.19059662){\color[rgb]{0,0,0}\makebox(0,0)[lt]{\lineheight{1.25}\smash{\begin{tabular}[t]{l}$\{0,1\}$\end{tabular}}}}%
    \put(0.69751485,0.44887186){\color[rgb]{0,0,0}\makebox(0,0)[lt]{\lineheight{1.25}\smash{\begin{tabular}[t]{l}$\{3\}$\end{tabular}}}}%
    \put(0.33272311,0.00603447){\color[rgb]{0,0,0}\makebox(0,0)[lt]{\lineheight{1.25}\smash{\begin{tabular}[t]{l}$\{0,1,2\}$\end{tabular}}}}%
    \put(-0.00094647,0.45595274){\color[rgb]{0,0,0}\makebox(0,0)[lt]{\lineheight{1.25}\smash{\begin{tabular}[t]{l}$\{-1\}$\end{tabular}}}}%
    \put(0.14578637,0.19059662){\color[rgb]{0,0,0}\makebox(0,0)[lt]{\lineheight{1.25}\smash{\begin{tabular}[t]{l}$\{-1,1\}$\end{tabular}}}}%
    \put(0.00129857,0.19059662){\color[rgb]{0,0,0}\makebox(0,0)[lt]{\lineheight{1.25}\smash{\begin{tabular}[t]{l}$\{-1,0\}$\end{tabular}}}}%
    \put(0.53796741,0.19059662){\color[rgb]{0,0,0}\makebox(0,0)[lt]{\lineheight{1.25}\smash{\begin{tabular}[t]{l}$\{1,2\}$\end{tabular}}}}%
    \put(0.68711607,0.19059662){\color[rgb]{0,0,0}\makebox(0,0)[lt]{\lineheight{1.25}\smash{\begin{tabular}[t]{l}$\{0,3\}$\end{tabular}}}}%
    \put(0.47721091,0.00603447){\color[rgb]{0,0,0}\makebox(0,0)[lt]{\lineheight{1.25}\smash{\begin{tabular}[t]{l}$\{0,1,3\}$\end{tabular}}}}%
    \put(0.83892812,0.19059662){\color[rgb]{0,0,0}\makebox(0,0)[lt]{\lineheight{1.25}\smash{\begin{tabular}[t]{l}$\{1,3\}$\end{tabular}}}}%
    \put(0.18771003,0.00603447){\color[rgb]{0,0,0}\makebox(0,0)[lt]{\lineheight{1.25}\smash{\begin{tabular}[t]{l}$\{-1,0,1\}$\end{tabular}}}}%
  \end{picture}%
\endgroup%

%% file: fig/Gamma2_hamming.pdf_tex
\begingroup%
  \makeatletter%
  \providecommand\color[2][]{%
    \errmessage{(Inkscape) Color is used for the text in Inkscape, but the package 'color.sty' is not loaded}%
    \renewcommand\color[2][]{}%
  }%
  \providecommand\transparent[1]{%
    \errmessage{(Inkscape) Transparency is used (non-zero) for the text in Inkscape, but the package 'transparent.sty' is not loaded}%
    \renewcommand\transparent[1]{}%
  }%
  \providecommand\rotatebox[2]{#2}%
  \newcommand*\fsize{\dimexpr\f@size pt\relax}%
  \newcommand*\lineheight[1]{\fontsize{\fsize}{#1\fsize}\selectfont}%
  \ifx\svgwidth\undefined%
    \setlength{\unitlength}{356.865739bp}%
    \ifx\svgscale\undefined%
      \relax%
    \else%
      \setlength{\unitlength}{\unitlength * \real{\svgscale}}%
    \fi%
  \else%
    \setlength{\unitlength}{\svgwidth}%
  \fi%
  \global\let\svgwidth\undefined%
  \global\let\svgscale\undefined%
  \makeatother%
  \begin{picture}(1,0.39222423)%
    \lineheight{1}%
    \setlength\tabcolsep{0pt}%
    \put(0.27350122,0.37093431){\color[rgb]{0,0,0}\makebox(0,0)[lt]{\lineheight{1.25}\smash{\begin{tabular}[t]{l}$\mathcal{P}=(\ldots,0,\check{0},0,0,0\ldots)$\end{tabular}}}}%
    \put(-0.00232605,0.22061468){\color[rgb]{0,0,0}\makebox(0,0)[lt]{\lineheight{1.25}\smash{\begin{tabular}[t]{l}$\mathcal{R}=(\ldots,0,\check{1},0,0,0,\ldots)$\end{tabular}}}}%
    \put(0.52850101,0.22175556){\color[rgb]{0,0,0}\makebox(0,0)[lt]{\lineheight{1.25}\smash{\begin{tabular}[t]{l}$(\ldots,0,\check{0},0,1,0\ldots)$\end{tabular}}}}%
    \put(0.30134783,0.00457008){\color[rgb]{0,0,0}\makebox(0,0)[lt]{\lineheight{1.25}\smash{\begin{tabular}[t]{l}$(\ldots,0,\check{1},0,1,0\ldots)$\end{tabular}}}}%
    \put(0,0){\includegraphics[width=\unitlength,page=1]{Gamma2_hamming.pdf}}%
  \end{picture}%
\endgroup%

%% file: fig/bstree.pdf_tex
\begingroup%
  \makeatletter%
  \providecommand\color[2][]{%
    \errmessage{(Inkscape) Color is used for the text in Inkscape, but the package 'color.sty' is not loaded}%
    \renewcommand\color[2][]{}%
  }%
  \providecommand\transparent[1]{%
    \errmessage{(Inkscape) Transparency is used (non-zero) for the text in Inkscape, but the package 'transparent.sty' is not loaded}%
    \renewcommand\transparent[1]{}%
  }%
  \providecommand\rotatebox[2]{#2}%
  \newcommand*\fsize{\dimexpr\f@size pt\relax}%
  \newcommand*\lineheight[1]{\fontsize{\fsize}{#1\fsize}\selectfont}%
  \ifx\svgwidth\undefined%
    \setlength{\unitlength}{556.28749096bp}%
    \ifx\svgscale\undefined%
      \relax%
    \else%
      \setlength{\unitlength}{\unitlength * \real{\svgscale}}%
    \fi%
  \else%
    \setlength{\unitlength}{\svgwidth}%
  \fi%
  \global\let\svgwidth\undefined%
  \global\let\svgscale\undefined%
  \makeatother%
  \begin{picture}(1,0.7236752)%
    \lineheight{1}%
    \setlength\tabcolsep{0pt}%
    \put(0,0){\includegraphics[width=\unitlength,page=1]{bstree.pdf}}%
    \put(-0.00072543,0.00964508){\color[rgb]{0,0,0}\makebox(0,0)[lt]{\lineheight{1.25}\smash{\begin{tabular}[t]{l}$G/G_{1}$\end{tabular}}}}%
    \put(-0.00128779,0.23053515){\color[rgb]{0,0,0}\makebox(0,0)[lt]{\lineheight{1.25}\smash{\begin{tabular}[t]{l}$G/G_{2}$\end{tabular}}}}%
    \put(-0.00152728,0.44876395){\color[rgb]{0,0,0}\makebox(0,0)[lt]{\lineheight{1.25}\smash{\begin{tabular}[t]{l}$G/G_{3}$\end{tabular}}}}%
    \put(-0.00109492,0.67408961){\color[rgb]{0,0,0}\makebox(0,0)[lt]{\lineheight{1.25}\smash{\begin{tabular}[t]{l}$G/G_{4}$\end{tabular}}}}%
    \put(0,0){\includegraphics[width=\unitlength,page=2]{bstree.pdf}}%
    \put(0.91907695,0.00423958){\color[rgb]{0,0,0}\makebox(0,0)[lt]{\lineheight{1.25}\smash{\begin{tabular}[t]{l}$G_{1}$\end{tabular}}}}%
    \put(0.91970427,0.23006062){\color[rgb]{0,0,0}\makebox(0,0)[lt]{\lineheight{1.25}\smash{\begin{tabular}[t]{l}$G_{2}$\end{tabular}}}}%
    \put(0.92095882,0.44709975){\color[rgb]{0,0,0}\makebox(0,0)[lt]{\lineheight{1.25}\smash{\begin{tabular}[t]{l}$G_{3}$\end{tabular}}}}%
    \put(0.91970427,0.68044816){\color[rgb]{0,0,0}\makebox(0,0)[lt]{\lineheight{1.25}\smash{\begin{tabular}[t]{l}$G_{4}$\end{tabular}}}}%
    \put(0,0){\includegraphics[width=\unitlength,page=3]{bstree.pdf}}%
    \put(0.80996066,0.11655659){\color[rgb]{0,0,0}\makebox(0,0)[lt]{\lineheight{1.25}\smash{\begin{tabular}[t]{l}$G_{1}$\end{tabular}}}}%
    \put(0.81136559,0.32983203){\color[rgb]{0,0,0}\makebox(0,0)[lt]{\lineheight{1.25}\smash{\begin{tabular}[t]{l}$G_{2}$\end{tabular}}}}%
    \put(0.81190955,0.55816219){\color[rgb]{0,0,0}\makebox(0,0)[lt]{\lineheight{1.25}\smash{\begin{tabular}[t]{l}$G_{3}$\end{tabular}}}}%
  \end{picture}%
\endgroup%